\documentclass[12pt]{article}
\usepackage[latin1]{inputenc}
\usepackage[T1]{fontenc}
\usepackage{authblk} 
\usepackage{amsmath}
\usepackage{amsfonts}
\usepackage{amssymb}
\usepackage[normalem]{ulem} 
\usepackage{amsmath,xcolor,ulem}
\usepackage{amsfonts}
\usepackage{amssymb}
\usepackage{amsthm}
\usepackage{graphicx}
\usepackage{marginnote}
\usepackage{subcaption}

\usepackage[colorlinks=true, allcolors=blue]{hyperref}
\usepackage[top=2.5cm,bottom=2.5cm,left=2.5cm,right=2.5cm]{geometry}
\usepackage{float}

\usepackage[algoruled]{algorithm2e}
\usepackage[font=footnotesize,width=\linewidth]{caption} 
\usepackage{graphicx}
\usepackage{amssymb}
\usepackage{amsthm}
\usepackage{amsmath}
\usepackage{lineno}
\usepackage{hyperref}
\usepackage{verbatim}

\newtheorem{theorem}{Theorem}[section]
\newtheorem{lemma}[theorem]{Lemma}
\theoremstyle{definition}

\theoremstyle{remark}
\newtheorem{remark}[theorem]{Remark}

\numberwithin{equation}{section}

\DeclareMathOperator{\trace}{tr}

\usepackage{graphicx, pdflscape}
\usepackage{amsmath, hyperref}
 \usepackage{amssymb} %  for mathbb
\usepackage{relsize, float, lineno}

 \usepackage{xcolor} % for colored comments
\usepackage{subcaption}
\allowdisplaybreaks
\usepackage[figcolor=white]{todonotes}

\title{Qualitative Analysis of a Pharmacokinetic Compartment Model for Human Ethanol Metabolism\\ and Its Generalization}
\author[1]{Manh Tuan Hoang\footnote{\href{mailto:tuanhm16@fe.edu.vn}{tuanhm16@fe.edu.vn(corresponding author)}}}
\affil[1]{Department of Mathematics, FPT University, Hoa Lac Hi-Tech Park, Km29 Thang Long Blvd, Hanoi, Viet Nam}
\begin{document}
\maketitle
\begin{abstract}
In this work, we revisit a continuous-time two-compartment pharmacokinetic model of human ethanol metabolism originally proposed by Levitt and Levitt. We first establish the positivity and boundedness of the solutions, investigate the existence and uniqueness of a positive equilibrium, and analyze its local and global asymptotic stability. As a result, the global dynamics of the ethanol metabolism model is completely characterized, thereby complementing and extending the analytical results reported in the original benchmark study.

{
Second, building upon the foundational two-compartment model of Levitt and Levitt, we replace the Michaelis--Menten metabolism rate with a general class of nonlinear functions, including both monotone and nonmonotone forms.
This generalization allows us to investigate a wider range of physiologically motivated scenarios and reveals richer dynamical behavior. In particular, for a specific nonmonotone family, we establish the occurrence of a nondegenerate
saddle-node bifurcation through center-manifold analysis. Furthermore, we prove the absence of nonconstant periodic orbits in the generalized model without assuming monotonicity. These findings extend the qualitative analysis beyond equilibrium stability and provide further insights into the complex dynamics of human ethanol metabolism.
}

Finally, numerical experiments are conducted to support the theoretical findings. The numerical results provide further evidence for the theoretical results. 
\end{abstract}
\begin{minipage}{0.9\linewidth}
%%%%%%%%%%%%%%%%%%%%%%%
\end{minipage}
%\linenumbers
%%%%%%%%%%%%%%%%%%%%%%%%%%%%%%%% Intro
\section{Introduction}\label{intro}
In an early and seminal study \cite{Levitt}, Levitt and Levitt proposed a two-compartment framework that accounts for the reduction in ethanol concentration as blood passes through the liver. In this model, the body-water and liver compartments are characterized by the ethanol concentrations $(C_B)$ and $(C_L)$, and the corresponding volumes of distribution $(V_B)$ and $(V_L)$, respectively. More clearly, the mathematical model is represented by a system of nonlinear ordinary differential equations of the form \cite{Levitt}:
\begin{equation}\label{eq:1}
\begin{split}
&V_B\dfrac{dC_B}{dt} = F_{H V} \left(C_L - C_B\right)+ Q_{I V},\\
&V_L\dfrac{dC_L}{dt} = F_{H V}\left(C_B - C_L\right) + Q_{G I}- V_{\max } \dfrac{C_L}{K_m + C_L},
\end{split}
\end{equation}
subject to initial data $C_B(0), C_L(0) \geq 0$. 
%\begin{itemize}
%\item  $C_B(t)$ and $C_L(t)$ denote the ethanol concentrations in the body-water and liver compartments, while $V_B$ and $V_L$ represent the corresponding ethanol distribution volumes;
%%%%%%%%%%%%%%%%%%%%%%%%%%%%%%%%%%%%%%%%
%\item the ethanol input rate $Q$ may occur either through intravenous administration at rate $Q_{IV}$ into the body-water compartment or through gastrointestinal absorption at rate $Q_{GI}$ directly into the liver compartment;
%%%
%%%
%\item $F_{HV}$ denotes the total hepatic blood flow, which is equal to the blood flow rate in the hepatic vein.
%\end{itemize}
%%
More details of the model \eqref{eq:1} was presented and discussed in \cite{Levitt}. {
For convenience, the physical meanings and
units of the variables and parameters
in the two-compartment model \eqref{eq:1} are
summarized in Table~\ref{tab:Parameter},
following Levitt and Levitt \cite{Levitt}.
\begin{table}[H]
\centering
\caption{Physical meanings and units of
the model variables and parameters}
\label{tab:Parameter}
\begin{tabular}{lll}
\hline
Symbol & Physical meaning & Unit\\
\hline
$C_B$ & Body-water ethanol concentration & mM\\
$C_L$ & Hepatic ethanol concentration & mM\\
$V_B$ & Body-water distribution volume & L\\
$V_L$ & Hepatic distribution volume & L\\
$F_{HV}$ & Hepatic blood flow rate & L/min\\
$Q_{IV}$ & Intravenous ethanol input rate & mmol/min\\
$Q_{GI}$ & Gastrointestinal ethanol input rate & mmol/min\\
$V_{\max}$ & Maximum ethanol metabolic rate & mmol/min\\
$K_m$ & Michaelis--Menten constant & mM\\
\hline
\end{tabular}
\end{table}
}

In \cite{Levitt}, the framework of \eqref{eq:1} was used to predict metabolic behavior over a wide range of blood alcohol levels. Consequently, the model \eqref{eq:1} is of considerable interest because understanding the relationship between blood ethanol concentration and hepatic metabolic rate is of both clinical and experimental importance, as it enables predictions of ethanol metabolism across a wide range of blood alcohol concentrations. To the best of our knowledge, the model \eqref{eq:1} has been extended in several subsequent studies to describe ethanol metabolism in the human body more realistically \cite{Kim, Villasanti, Villasanti1, Wacker1, Wacker2, Whitmire, Zekan}. In particular, Wacker proposed in \cite{Wacker1, Wacker2} an extended version of \eqref{eq:1} and investigated its qualitative properties through rigorous mathematical analysis and dynamically consistent numerical schemes.

Despite its biological significance, the dynamical behavior of the model \eqref{eq:1} has not yet been investigated from the perspective of dynamical systems. The dynamical analysis provides {a priori} information for the ethanol metabolism process and offers valuable insights into its long-term behavior, both of which are useful for practical applications. Motivated by the above considerations, the first part of this work is devoted to a rigorous mathematical analysis of the dynamical properties of \eqref{eq:1}. In particular, we establish the positivity and boundedness of solutions, identify all equilibrium points, and analyze their local and global asymptotic stability (LAS and GAS). These theoretical results reveal the rich dynamics of the model and provide a mathematical foundation for its application to real-world problems.

In the second part of this work, we propose a generalized version of model \eqref{eq:1} and analyze the dynamics of the generalized model. To end this, let us consider the Michaelis--Menten function having the form
\begin{equation}\label{eq:2}
f(x) := \dfrac{\kappa_1 x}{1 + \kappa_2 x}, \quad \kappa_1, \kappa_2 > 0
\end{equation}
which appears in \eqref{eq:1} to describe the hepatic ethanol metabolism rate. This function characterizes the saturation effect \cite{Capasso} and has been widely adopted in epidemiological models (see, for example, \cite{Cui, Xu, Zhang} and the references therein). Note that the saturated function in the form \eqref{eq:2} is monotonically increasing. Other examples of monotonically increasing saturating functions can be represented by
\begin{equation}\label{eq:2a}
f(x) = \kappa_1\left(1 - e^{-\kappa_2 x}\right), \quad \kappa_1, \kappa_2 > 0,
\end{equation}
and
\begin{equation}\label{eq:2b}
f(x) = \dfrac{\kappa_1 x^m}{\kappa_2 + x^m}, \quad \kappa_1, \kappa_2 > 0, \quad m \geq 1.
\end{equation}
%%%

%
%
In \cite{Ruan, Xiao}, nonmonotone incidence functions have also been adopted in epidemic models as a suitable and effective alternative to the monotone incidence rates. As a particular case, $f$ may be assumed to attain its maximum value at a threshold $x=x^*$, increasing gradually for $x<x^*$ and decreasing gradually for $x>x^*$.  Such a nonmonotone response can be interpreted as a psychological (inhibitory or overload) effect \cite{Ruan, Xiao}, whereby an excessively high concentration reduces the effective metabolism rate. A commonly used function that captures this behavior is given by \cite{Xiao}
\begin{equation}\label{eq:2c}
f(x) := \dfrac{\kappa_1 x}{\kappa_2 + x^2}, \quad \kappa_1, \kappa_2 > 0,
\end{equation}
or more generally
\begin{equation}\label{eq:4}
f(x) = \dfrac{\kappa_1 x^n}{\kappa_2 + x^m}, \quad m > n \geq 1, \quad \kappa_1, \kappa_2 > 0.
\end{equation}
The graphs of two functions $f$ given in \eqref{eq:2} and \eqref{eq:4} are depicted in Figure \ref{Fig:1f}.

\begin{figure}[H]
\includegraphics[height=9.0cm,width=17cm]{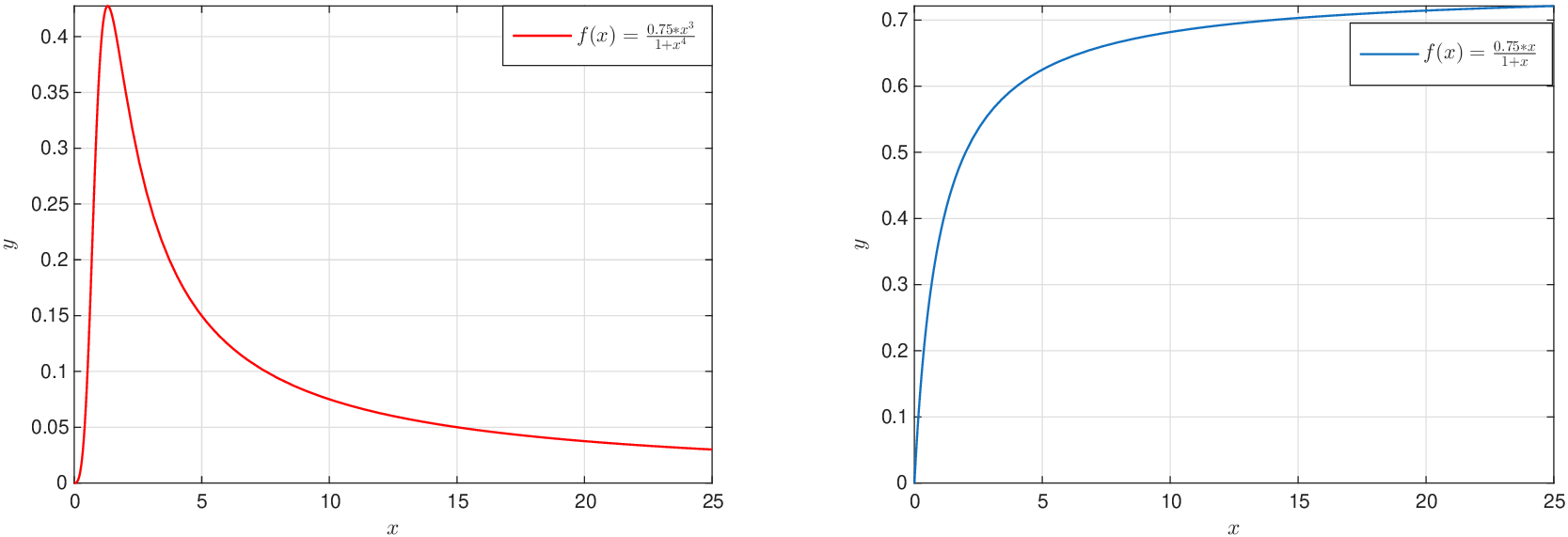}
\caption{Nonmonotone (left) and monotone (right) rate functions. The nonmonotone rate function $f(x) = \frac{0.75x^3}{1 + x^4}$ represents the overload effect, whereas the monotone rate function $f(x) = \frac{0.75x}{1 + x}$ represents the saturation effect.}\label{Fig:1f}
\end{figure}

To provide a more realistic and flexible mathematical model of hepatic ethanol metabolism, we generalize the Michaelis--Menten rate function in \eqref{eq:1} by introducing a broader class of nonlinear rate functions satisfying the following property:\\
{
\textbf{(A)}: $f\in C^1([0,\infty))$;\quad $f(0) = 0$;\quad $f(x) > 0$ for all $x > 0$.\\
Assumption \textbf{(A)}, combined with the positive invariance
of the nonnegative quadrant and the mass-balance estimate,
ensures the existence and uniqueness of nonnegative solutions
for all $t\geq0$. The continuous differentiability of $f$
also permits the subsequent linearization analysis.
}

Thus, the original model \eqref{eq:1} is generalized as follows:
\begin{equation}\label{eq:1new}
\begin{split}
&V_B\dfrac{dC_B}{dt} = F_{H V} \left(C_L - C_B\right)+ Q_{I V},\\
&V_L\dfrac{dC_L}{dt} = F_{H V}\left(C_B - C_L\right) + Q_{G I}- f(C_L),
\end{split}
\end{equation}
where $f$ is any function satisfying the property \textbf{(A)}.

It is important to note that the class of functions satisfying the property \textbf{(A)} represents the most general class of metabolization rate functions compatible with the underlying biological assumptions. Consequently, model \eqref{eq:1new} provides a unified framework capable of describing a wide range of realistic metabolic scenarios through an appropriate choice of the function $f$. In particular, the class of nonlinear functions $f$ satisfying the property \textbf{(A)} includes several well-known examples, such as the monotonically increasing saturating functions in \eqref{eq:2a} and \eqref{eq:2b}, the nonmonotone incidence function given by \eqref{eq:4}, and the Holling type I, II, III, and IV functional responses \cite{Dawes1,Wu1}. More general functional responses satisfying the property \textbf{(A)} are introduced in \cite{Kalinkat1}. 

Since hepatic ethanol elimination is regulated by various physiological factors, including the availability and activity of alcohol dehydrogenase \cite{Edenberg1}, the enzymatic capacity of the liver is finite. Therefore, it is natural to assume that $f$ is bounded above by a positive constant. Moreover, the assumption that the ethanol elimination rate grows at most linearly, i.e., $f(C_L) \leq MC_L$ for some $M > 0$, is biologically motivated because ethanol elimination is an enzyme-mediated process with finite catalytic capacity. Consequently, the elimination rate cannot increase faster than proportionally to the ethanol concentration and is satisfied by many commonly used kinetic functions, including the Michaelis--Menten rate.

{In a recent study \cite{HoangNgoWacker}, the class of nonlinear functions having the property \textbf{(A)} has been used in a three-compartment model of ethanol metabolism in the human body. However, as shown in the following sections, the analysis of the generalized model \eqref{eq:1new} is considerably more challenging. The main difference between the present study and \cite{HoangNgoWacker} is the
presence of persistent external ethanol inputs, $Q_{IV}$ and
$Q_{GI}$. In \cite{HoangNgoWacker}, no such inputs are considered, the total
ethanol amount is nonincreasing, and the trivial equilibrium
is globally asymptotically stable. In contrast, our model satisfies
\begin{equation*}
\dfrac{d}{dt}(V_BC_B+V_LC_L) =Q_{IV}+Q_{GI}-f(C_L),
\end{equation*}
which means that the total ethanol amount need not decrease. Despite having fewer compartments, the present model exhibits richer dynamics than the model in \cite{HoangNgoWacker}. For the non-monotone
family considered in Theorem~\ref{Theorem5}, varying the total input rate
yields two positive equilibria with different stability
properties, a unique non-hyperbolic equilibrium, or no
equilibrium. The transition occurs through a saddle-node
bifurcation. These input-dependent regimes differ qualitatively
from the global convergence to the trivial equilibrium
established in \cite{HoangNgoWacker} and constitute a central distinction
between the two studies.}

Based on a rigorous mathematical analysis, we establish the dynamical properties of the generalized model, including the positivity and boundedness of the solutions, the set of possible equilibrium points and their local and global asymptotic stability. The results reveal the rich dynamical behavior of the generalized model and provide insights into its potential real-world applications. A key step in the the dynamical analysis of both models \eqref{eq:1} and \eqref{eq:1new} is the construction of a quadratic Lyapunov function, which allows us to establish the global asymptotic stability of the unique positive equilibrium points. This Lyapunov-based approach is sufficiently general to be adapted to discrete-time and fractional-order versions. It is worth emphasizing that constructing a suitable Lyapunov function for a nonlinear dynamical system is, in general, a nontrivial and challenging problem \cite{Cangiotti, Korobeinikov1, Korobeinikov2, Vargas}. {In particular, for a specific nonmonotone family, we establish the occurrence of a nondegenerate
saddle-node bifurcation through center-manifold analysis. Furthermore, we prove the absence of nonconstant periodic orbits in the generalized model without assuming monotonicity. These findings extend the qualitative analysis beyond equilibrium stability and provide further insights into the complex dynamics of human ethanol metabolism.}

{
The two-compartment model originally proposed by Levitt and Levitt provides a foundational framework for studying human ethanol metabolism. Extending this framework and investigating its dynamics are therefore of both theoretical and practical interest. In this study, we incorporate a general class of nonlinear metabolic functions to explore a broader range of physiologically motivated scenarios. For the non-monotone family considered, the model exhibits richer dynamics, including multiple equilibria with different stability properties and a saddle-node bifurcation. Our qualitative results clarify how external ethanol inputs and metabolic capacity influence ethanol accumulation and elimination, and may inform experimental design, parameter estimation, and the interpretation of physiological observations. These findings also provide a mathematical foundation for future extensions incorporating delays, stochastic effects, fractional-order dynamics, and data-driven approaches.}
%%%%%%%%%%%%%%%%%%%%%%%%%%%%
%%%%%%%%%%%%%%%%%%%%%%%%%%%%%

Along with the theoretical analysis, we conduct a series of numerical simulations using different nonlinear functions to illustrate the theoretical results under a variety of representative scenarios. The numerical simulation results reveal some open problems in the numerical analysis of the models under consideration, which deserve further investigation.

The findings of this work provide an important theoretical foundation for the quantitative analysis of ethanol clearance in the human body, thereby providing a basis for a wide range of practical applications.

%%%%%%%%%%%%%%%%%%%%%%%%%%%%%%%%%%%%%%%%%%%%%%%%%%%%%%%%%%%%%%%%%%%%

The organization of this paper is as follows:\\
Section \ref{Sec2} investigates the dynamical properties of the original model \eqref{eq:1}. The analysis of the generalized model \eqref{eq:1new} is presented in Section \ref{Sec3}. Section \ref{Sec4} presents a series of numerical experiments to support the theoretical findings. Finally, the last section concludes the paper with some concluding remarks and a discussion of some open problems.
\section{Dynamical analysis of the model with Michaelis--Menten mechanism}\label{Sec2}
In this section, we analyze the dynamical properties of the original model \eqref{eq:1}.
%\subsection{Analysis of the continuous-time model}
%\subsection{Basic properties}
First,  we establish the positivity and boundedness of the solutions.
%%
%%%%%%%%%%%%%
%%%%%%%%%%%%%
\begin{theorem}\label{Theorem1}
The model \eqref{eq:1} admits the closed first quadrant in $\mathbb{R}^2$, $\mathbb{R}_+^2 = \{(x,\,y)|x, y \geq 0\}$, as a positively invariant set. Moreover, if $C_B(0), C_L(0) \geq 0$, then $C_B(t), C_L(t) > 0$ for all $t > 0$.
\end{theorem}
%%%
\begin{proof}
First, it immediately follows from the system \eqref{eq:1} that
\begin{equation}\label{eq:5}
\begin{split}
&\dfrac{dC_B}{dt}\bigg|_{C_B = 0} = \dfrac{F_{H V}}{V_B} C_L + \dfrac{Q_{I V}}{V_B} \geq 0,\\
&\dfrac{dC_L}{dt}\bigg|_{C_L = 0} = \dfrac{F_{H V}}{V_L} C_B + \dfrac{Q_{G I}}{V_L} \geq 0
\end{split}
\end{equation}
for all $C_B, C_L \geq 0$. As a direct consequence of \cite[Proposition B.7]{Smith}, we conclude that $C_B(t), C_L(t) \geq 0$ for $t > 0$ whenever $C_B(0), C_L(0) \geq 0$. This is the desired conclusion. The proof is complete.

Let $\left(C_B(0),\,C_L(0)\right) \in \mathbb{R}_+^2$ be any initial data. If $C_B(0) = 0$, then \eqref{eq:5} implies that  
\begin{equation*}
\dfrac{dC_B}{dt}\bigg|_{t = 0} = \dfrac{F_{H V}}{V_B} C_L(0) + \dfrac{Q_{I V}}{V_B} > 0,
\end{equation*}
which implies that there exists $t_0 > 0$ such that $C_B(t_0) > 0$. Hence, without the loss of generality, we can assume that $C_B(0) > 0$. Assume that there exists $t_1 > 0$ such that $C_B(t_1) = 0$, let us denote
\begin{equation*}
t_* = \min\{s| C_B(s) = 0\}.
\end{equation*}
At $t = t_*$, we have
\begin{equation*}
\dfrac{dC_B}{dt}\bigg|_{t = t_*} = \dfrac{F_{H V}}{V_B} C_L(t_*) + \dfrac{Q_{I V}}{V_B} > 0.
\end{equation*}
Hence, the continuity of $C(t)$ implies that there exists $\epsilon > 0$ such that
\begin{equation*}
     C_B'(t) > 0, \quad t \in (t_* - \epsilon,\,t_* + \epsilon) \subset (0,\,t_*+\epsilon).
\end{equation*}
Thus, for $t_* - \epsilon < t < t_*$, we get
\begin{equation*}
    C_B(t) < C_B(t_*) = 0.
\end{equation*}
This is a contradiction to $C_B(t) \geq 0$ for $t > 0$.

Repeating the above arguments, we obtain $C_L(t) > 0$ for $t> 0$. The proof is complete.
\end{proof}
We now determine the set of equilibrium point of \eqref{eq:1}.
\begin{lemma}\label{Lemma1}
The model \eqref{eq:1} possesses a unique positive equilibrium point $E^* = \left(C_B^*,\,C_L^*\right)$ if and only if
\begin{equation}\label{eq:6}
V_{max} > Q_{IV} + Q_{GI}.
\end{equation}
Moreover, when this is the case, $E^*$ is determined by
\begin{equation}\label{eq:7}
\begin{split}
&{C_L^* = \dfrac{K_{m}\left(Q_{IV} + Q_{GI}\right)}{V_{max}- \left(Q_{IV} + Q_{GI}\right)}},\\
&C_B^* = \dfrac{Q_{IV}}{F_{HV}} + C_L^*.
\end{split}
\end{equation}

\end{lemma}
\begin{proof}
Any equilibrium point of \eqref{eq:1} is a solution to the system
\begin{equation}\label{eq:8}
\begin{split}
&F_{H V} \left(C_L - C_B\right)+ Q_{I V} = 0,\\
&F_{H V}\left(C_B - C_L\right) + Q_{G I}- V_{\max } \dfrac{C_L}{K_m + C_L} = 0.
\end{split}
\end{equation}
Adding side-by-side the two equations of \eqref{eq:8} gives
\begin{equation*}
Q_{IV} + Q_{G I}- V_{\max } \dfrac{C_L}{K_m + C_L} =0.
\end{equation*}
This equation has a unique positive solution, which is defined by the first formula of \eqref{eq:7}, if and only if \eqref{eq:6} holds. Using the first equation of \eqref{eq:8} leads to the second formula for $C_B^*$ in \eqref{eq:7}. The proof is completed.
\end{proof}
%%
%%%
\begin{remark}
The condition \eqref{eq:6} means that the maximum metabolic capacity of the liver exceeds the total ethanol input rate. In other words, the liver is capable of metabolizing ethanol at a rate greater than the combined intravenous infusion rate and gastrointestinal absorption rate. Consequently, ethanol cannot accumulate indefinitely in the body, and the system admits a finite steady state.

When \eqref{eq:6} is not satisfied, that is
\begin{equation*}
V_{max} \leq Q_{IV} + Q_{GI},
\end{equation*}
we define the function
\begin{equation*}
S(t) = V_B C_B(t) + V_LC_L(t).
\end{equation*}
Then, it follows from \eqref{eq:1} that
\begin{equation*}
\dfrac{dS}{dt} = Q_{I V} + Q_{G I}- V_{\max } \dfrac{C_L}{K_m + C_L} > Q_{I V} + Q_{G I}- V_{\max },
\end{equation*}
which implies that the total amount of ethanol increases monotonically and no finite equilibrium can exist. In particular, from a basic comparison theorem for ODEs \cite{McNabb}, we obtain the estimate
\begin{equation*}
S(t) > \left(Q_{I V} + Q_{G I}- V_{\max}\right)t + S(0).
\end{equation*}
\end{remark}
%
%%
%%
%\subsection{Stability analysis}
%%
%%
%%
To end this section, we establish the LAS and GAS of the unique equilibrium point of \eqref{eq:1} whenever it exists.
\begin{theorem}[Stability analysis]\label{Theorem2}
If the unique positive equilibrium point $E^*$ of \eqref{eq:1} exists, then it is not only locally asymptotically stable but also globally asymptotically stable.
\end{theorem}
\begin{proof}
In order to analyze the LAS, we examine the Jacobian matrix of \eqref{eq:1} evaluated at $E^*$, which is given by
\begin{equation*}
J_{\eqref{eq:1}}(E^*) =
\begin{pmatrix}
-\dfrac{F_{HV}}{V_B}&\dfrac{F_{HV}}{V_B},\\
%%
%%%
\dfrac{F_{HV}}{V_L}& -\dfrac{F_{HV}}{V_L} - \dfrac{1}{V_L}\dfrac{V_{max}K_m}{\left(K_m + C_L^*\right)^2}
\end{pmatrix}.
\end{equation*}
By simple algebraic manipulations, we obtain
\begin{equation*}
\begin{split}
&\trace(J_{\eqref{eq:1}}(E^*)) = -\dfrac{F_{HV}}{V_B} -\dfrac{F_{HV}}{V_L} - \dfrac{1}{V_L}\dfrac{V_{max}K_m}{\left(K_m + C_L^*\right)^2} < 0,\\
&\det(J_{\eqref{eq:1}}(E^*)) = \dfrac{1}{V_L}\dfrac{F_{HV}}{V_B}\dfrac{V_{max}K_m}{\left(K_m + C_L^*\right)^2} > 0.
\end{split}
\end{equation*}
By the Routh--Hurwitz criterion (see \cite[Theorem 4.4]{Allen}), all eigenvalues of $J_{\eqref{eq:1}}(E^*)$ have negative real parts. Therefore, the linearization principle \cite{Khalil, Stuart} implies that $E^*$ is locally asymptotically stable.

To show the GAS of $E^*$, we consider a Lyapunov function candidate $W: int(\mathbb{R}_+^2) \to \mathbb{R}_+$ defined by
\begin{equation}\label{eq:9}
W(C_B,\,C_L) = \dfrac{1}{2}V_B\left(C_B - C_B^*\right)^2 + \dfrac{1}{2}V_L\left(C_L - C_L^*\right)^2.
\end{equation}
%%
%%%
Since $E^*$ is the unique positive equilibrium point, we use \eqref{eq:8} to rewrite \eqref{eq:1} in the form:
\begin{equation}\label{eq:1GAS}
\begin{split}
&V_B\dfrac{dC_B}{dt} = -F_{H V}\left(C_B - C_B^*\right)  + F_{H V}\left(C_L - C_L^*\right),\\
&V_L\dfrac{dC_L}{dt} = F_{H V}\left(C_B - C_B^*\right) -  F_{H V}\left(C_L - C_L^*\right) - \left(V_{\max } \dfrac{C_L}{K_m + C_L} - V_{\max } \dfrac{C_L^*}{K_m + C_L^*}\right),
\end{split}
\end{equation}
Hence, the derivative of $W$ along with the solutions of \eqref{eq:1GAS} is given by
\begin{equation}\label{eq:10}
\begin{split}
\dfrac{dW}{dt} &= \dfrac{dW}{dC_B}\dfrac{dC_B}{dt} + \dfrac{dW}{dC_L}\dfrac{dC_L}{dt}\\
%%%%%%%%%%%%%%
&= \left(C_B - C_B^*\right)\left[-F_{H V}\left(C_B - C_B^*\right)  + F_{H V}\left(C_L - C_L^*\right)\right]\\
%%%
&+ \left(C_L - C_L^*\right)\left[F_{H V}\left(C_B - C_B^*\right) -  F_{H V}\left(C_L - C_L^*\right) - \left(V_{\max } \dfrac{C_L}{K_m + C_L} - V_{\max } \dfrac{C_L^*}{K_m + C_L^*}\right)\right]\\
&= -F_{HV}\left(C_B - C_B^*\right)^2 - F_{HV}\left(C_L - C_L^*\right)^2 + 2F_{HV}\left(C_B - C_B^*\right)\left(C_L - C_L^*\right)\\
&- \left(V_{\max } \dfrac{C_L}{K_m + C_L} - V_{\max } \dfrac{C_L^*}{K_m + C_L^*}\right)\left(C_L - C_L^*\right)\\
&={-F_{HV}\big[\left(C_B - C_B^*\right) - \left(C_L- C_L^*\right)\big]^2} - \left(V_{\max } \dfrac{C_L}{K_m + C_L} - V_{\max } \dfrac{C_L^*}{K_m + C_L^*}\right)\left(C_L - C_L^*\right).
\end{split}
\end{equation}
Since the Michaelis--Menten  function is increasing, \eqref{eq:10} implies that $\frac{dW}{dt} \leq 0$ for all $C_B, C_L \geq 0$ and $\frac{dW}{dt} = 0$ if and only if $\left(C_B,\,C_L\right) = E^*$. Using Lyapunov's direct method \cite{Allen, Khalil, Stuart}, the GAS $E^*$ is proved. The proof is complete.
\end{proof}

\begin{remark}
The equilibrium point $E^*$ can be referred to as the ethanol-present equilibrium point. The analysis given in this section remains valid when either $Q_{IV}=0$ or $Q_{GI}=0$. In particular, when $Q_{IV}=Q_{GI}=0$, that is, after ethanol input has stopped, the model \eqref{eq:1} is reduced to
\begin{equation}\label{eq:1S}
\begin{split}
&V_B\dfrac{dC_B}{dt} = F_{H V} \left(C_L - C_B\right),\\
&V_L\dfrac{dC_L}{dt} = F_{H V}\left(C_B - C_L\right) - V_{\max } \dfrac{C_L}{K_m + C_L}.
\end{split}
\end{equation}
Then, the equilibrium point $E^*$ becomes the origin and is globally asymptotically stable. The GAS of the origin implies the complete elimination of ethanol from the body.  This behavior is fully consistent with physiological reality.
\end{remark}
\section{Dynamical analysis of the generalized model}\label{Sec3}
This section is devoted to analyzing the dynamical properties of the generalized model \eqref{eq:1new}. 

First, based on the arguments used in Theorem \ref{Theorem1}, we conclude that: The generalized model \eqref{eq:1new} also admits the closed first quadrant in $\mathbb{R}^2$, $\mathbb{R}_+^2 = \{(x,\,y)|x, y \geq 0\}$, as a positively invariant set. Moreover, $C_B(t), C_L(t) > 0$ for all $t > 0$ whenever $C_B(0), C_L(0) \geq 0$.

{
Let $(C_B(t),C_L(t))$ be a solution with initial
data $C_B(0), C_L(0) \geq 0$. We define the total ethanol amount by
\begin{equation*}
M(t) = V_BC_B(t)+V_LC_L(t).
\end{equation*}
It follows from \eqref{eq:1new} that
\begin{equation*}
\dot M(t)=Q_{IV}+Q_{GI}-f(C_L(t))
\leq Q_{IV}+Q_{GI},
\end{equation*}
which implies that
\begin{equation*}
0\leq M(t)\leq M(0)+(Q_{IV}+Q_{GI})t
\end{equation*}
throughout the maximal forward interval of existence.
Since $V_B,V_L>0$, this estimate prevents either concentration
from becoming unbounded in finite time. Consequently,
under Assumption \textbf{(A)}, the unique local solution
extends to all $t\geq0$ (see Corollary~2 in Section~2.4 of Perko
\cite[p.~91]{Perko2001}).
}

Any equilibrium point of \eqref{eq:1new} is a solution to the system
\begin{equation*}
\begin{split}
&F_{H V} \left(C_L - C_B\right)+ Q_{I V} = 0,\\
&F_{H V}\left(C_B - C_L\right) + Q_{G I}- f(C_L) = 0,
\end{split}
\end{equation*}
which is equivalent to
\begin{equation}\label{eq:11}
\begin{split}
&C_B = \dfrac{Q_{IV}}{F_{HV}} + C_L,\\
&f(C_L) = Q_{I V} + Q_{GI}.
\end{split}
\end{equation}
Therefore, the number of equilibrium points of \eqref{eq:1new} is determined by the number of positive solutions of the second equation in \eqref{eq:11}. Consequently, \eqref{eq:1new} may admit no equilibrium point, a unique equilibrium point, or multiple equilibrium points. In particular, if $f$ is an increasing function as the Michaelis--Menten function, we obtain:
\begin{lemma}\label{Lemma2}
If $f$ is an increasing function for $C_L > 0$ and 
\begin{equation*}
\lim_{C_L \to \infty}f(C_L) > Q_{I V} + Q_{GI},
\end{equation*}
then \eqref{eq:1new} possesses a unique positive equilibrium point, which is determined by \eqref{eq:11}.
\end{lemma}
We now analyze the LAS of existing equilibrium points of \eqref{eq:1new}.
\begin{theorem}\label{Theorem3}
Let $E^* = (C_B^*,\,C_L^*)$ be any positive equilibrium point of \eqref{eq:1new}, which is determined by \eqref{eq:11}. Then, $E^*$ is locally asymptotically stable if $f'(C_L^*) > 0$ and is unstable if $f'(C_L^*) < 0$.
\end{theorem}
\begin{proof}
The Jacobian matrix of the system \eqref{eq:1new} evaluated at $E^*$ is given by
\begin{equation*}
J_{\eqref{eq:1new}}(E^*) =
\begin{pmatrix}
-\dfrac{F_{HV}}{V_B}&\dfrac{F_{HV}}{V_B},\\
%%
%%%
\dfrac{F_{HV}}{V_L}& -\dfrac{F_{HV}}{V_L} - \dfrac{1}{V_L}f'(C_L^*)
\end{pmatrix}.
\end{equation*}
Consequently,
\begin{equation*}
\begin{split}
&\trace(J_{\eqref{eq:1new}}(E^*)) = -\dfrac{F_{HV}}{V_B} -\dfrac{F_{HV}}{V_L} - \dfrac{1}{V_L}f'(C_L^*),\\
&\det(J_{\eqref{eq:1new}}(E^*)) = \dfrac{F_{HV}}{V_B}\dfrac{1}{V_L}f'(C_L^*).
\end{split}
\end{equation*}
If $f'(C_L^*) > 0$, then $\trace(J_{\eqref{eq:1new}}(E^*)) < 0$ and $\det(J_{\eqref{eq:1new}}(E^*)) > 0$. This implies that all eigenvalues of $J_{\eqref{eq:1new}}(E^*)$ have negative real parts \cite[Theorem 4.4]{Allen}, and therefore, $E^*$ is locally asymptotically stable according to the linearization principle \cite{Khalil, Stuart}.

If $f'(C_L^*) < 0$, {then $\det(J_{\eqref{eq:1new}}(E^*)) < 0$}. This implies that $J_{\eqref{eq:1new}}(E^*)$ has an eigenvalue that has positive real part \cite[Theorem 4.4]{Allen}. By using  the linearization principle \cite{Khalil, Stuart}, we conclude that $E^*$ is unstable.

The proof is complete.
\end{proof}
Our next investigation is focused on the GAS of \eqref{eq:1new}. Compared with \eqref{eq:1}, the GAS analysis of \eqref{eq:1new} is considerably more involved due to the more complex model structure. The following result, which generalizes Theorem \ref{Theorem2}, provides a sufficient condition under which \eqref{eq:1new} has a unique positive equilibrium point that is globally asymptotically stable.
%%
%%%
\begin{theorem}\label{Theorem4}
Assume that $f$ is {strictly increasing function} with the property that
\begin{equation*}
\lim_{C_L \to \infty}f(C_L) > Q_{I V} + Q_{GI}.
\end{equation*}
Then, the model \eqref{eq:1new} has a unique positive equilibrium point that is not only locally asymptotically stable but also globally asymptotically stable.
\end{theorem}
\begin{proof}
First, from Lemma \ref{Lemma1}, we conclude that \eqref{eq:1new} has a positive equilibrium point, namely $E^* = \left(C_B^*,\,C_L^*\right)$. On the other hand, Theorem \ref{Theorem3} implies that $E^*$ is locally asymptotically stable.

To establish the GAS of $E^*$, we consider a Lyapunov function candidate $\mathcal{W}: int(\mathbb{R}_+^2) \to \mathbb{R}_+$ defined by
\begin{equation}\label{eq:12}
\mathcal{W}(C_B,\,C_L) = \dfrac{1}{2}V_B\left(C_B - C_B^*\right)^2 + \dfrac{1}{2}V_L\left(C_L - C_L^*\right)^2.
\end{equation}
%%
%%%
Since $E^*$ is the unique positive equilibrium point, \eqref{eq:1new} can be represented in the form:
\begin{equation}\label{eq:2GAS}
\begin{split}
&V_B\dfrac{dC_B}{dt} = -F_{H V}\left(C_B - C_B^*\right)  + F_{H V}\left(C_L - C_L^*\right),\\
&V_L\dfrac{dC_L}{dt} = F_{H V}\left(C_B - C_B^*\right) -  F_{H V}\left(C_L - C_L^*\right) - \left(f(C_L) - f(C_L^*)\right),
\end{split}
\end{equation}
Hence, the derivative of $\mathcal{W}$ along with the solutions of \eqref{eq:2GAS} is given by
\begin{equation}\label{eq:13}
\begin{split}
\dfrac{d\mathcal{W}}{dt} &= \dfrac{d\mathcal{W}}{dC_B}\dfrac{dC_B}{dt} + \dfrac{d\mathcal{W}}{dC_L}\dfrac{dC_L}{dt}\\
%%%%%%%%%%%%%%
&= \left(C_B - C_B^*\right)\left[-F_{H V}\left(C_B - C_B^*\right)  + F_{H V}\left(C_L - C_L^*\right)\right]\\
%%%
&+ \left(C_L - C_L^*\right)\left[F_{H V}\left(C_B - C_B^*\right) -  F_{H V}\left(C_L - C_L^*\right) - \left(f(C_L)- f(C_L^*)\right)\right]\\
&= -F_{HV}\left(C_B - C_B^*\right)^2 - F_{HV}\left(C_L - C_L^*\right)^2 + 2F_{HV}\left(C_B - C_B^*\right)\left(C_L - C_L^*\right)\\
&- \left(f(C_L) - f(C_L^*)\right)\left(C_L - C_L^*\right)\\
&={-F_{HV}\big[\left(C_B - C_B^*\right) - \left(C_L - C_L^*\right)\big]^2} - \left(f(C_L) - f(C_L^*)\right)\left(C_L - C_L^*\right).
\end{split}
\end{equation}
Because $f$ is increasing, we deduce from \eqref{eq:13} that $\frac{d\mathcal{W}}{dt} \leq 0$ for all $C_B, C_L \geq 0$ and $\frac{d\mathcal{W}}{dt} = 0$ if and only if $\left(C_B,\,C_L\right) = E^*$. Hence, the GAS $E^*$ is obtained thanks to Lyapunov's direct method \cite{Allen, Khalil, Stuart}. The proof is complete.
\end{proof}
\begin{remark}
Since the saturating functions defined in \eqref{eq:2a} and \eqref{eq:2b} are monotonically increasing, the results of this section can be directly applied to establish the global dynamics of model \eqref{eq:1new} with these functions.
\end{remark}
To end this section, we consider a special case of the model \eqref{eq:1new}, which uses the hepatic ethanol metabolism rate given in \eqref{eq:4}. Specifically, the model under consideration is given by
\begin{equation}\label{eq:16}
\begin{split}
&V_B\dfrac{dC_B}{dt} = F_{H V} \left(C_L - C_B\right)+ Q_{I V},\\
&V_L\dfrac{dC_L}{dt} = F_{H V}\left(C_B - C_L\right) + Q_{G I} - \dfrac{\kappa_1 C_L^n}{\kappa_2 + C_L^m},
\end{split}
\end{equation}
where $\kappa_1, \kappa_2, m$ and $n$ are positive real numbers with $m > n \geq 1$.

Equilibrium points of \eqref{eq:16} are solutions to the system
\begin{equation}\label{eq:17a}
\begin{split}
&C_B = \dfrac{Q_{IV}}{F_{HV}} + C_L,\\
&\dfrac{\kappa_1 C_L^n}{\kappa_2 + C_L^m} = Q_{I V} + Q_{GI}.
\end{split}
\end{equation}
As analyzed before, the function $f(C_L) = \frac{\kappa_1 C_L^n}{\kappa_2 + C_L^m}$  can be interpreted as a psychological (inhibitory or overload) effect.  On the interval $[0,\,\,\infty)$, it attains a maximum value at $\widehat{C}_L = \left(\frac{n\kappa_2}{m - n}\right)^{\frac{1}{m}}$ with
\begin{equation*}\label{eq:17}
\widehat{V}_{max} := \max_{C_L > 0}f(C_L) = f(\widehat{C}_L) = {f_{max} = \dfrac{\kappa_1(m-n)}{m\kappa_2}
\left(
\dfrac{n\kappa_2}{m-n}
\right)^{\frac{n}{m}}}.
\end{equation*}
%%%%%%%%%%%%%%%%%%%%%%%%%%%%%%%%%%%%%
Combining this with Theorem \ref{Theorem3} leads to the following result.
\begin{theorem}[Stability analysis of the model \eqref{eq:16}]\label{Theorem5}
The following assertions hold for the model \eqref{eq:16}:
\begin{itemize}
\item[(i)] If $Q_{I V} + Q_{GI} > \widehat{V}_{max}$, then \eqref{eq:16} has no positive equilibrium point.
%%%
%%
%\item[(ii)] If $Q_{I V} + Q_{GI} = \widehat{V}_{max}$, then \eqref{eq:16} has a unique positive equilibrium point $E^* = \left(C_B^*,\,C_L^*\right)$, which is defined by \eqref{eq:17a}.
%Moreover, this equilibrium point is non-hyperbolic.
{
\item[(ii)] If $Q_{IV}+Q_{GI}=\widehat{V}_{max}$,
then \eqref{eq:16} has a unique positive
equilibrium point $E^*=(C_B^*,C_L^*)$,
defined by \eqref{eq:17a}.
This equilibrium is a non-hyperbolic
saddle-node and is unstable in the sense
of Lyapunov. Furthermore, the system
undergoes a nondegenerate saddle-node
bifurcation as the total ethanol input
rate passes through $\widehat{V}_{max}$.
}
%%%%%%%%%%%%%%%%%%%%%%%%%%%%%%%%%%%%%%%%%
%%%%%%%%%%%%%%%%%%%%%%%%%%%%%%%%%%%%%%
\item[(iii)] If $Q_{I V} + Q_{GI} < \widehat{V}_{max}$, then \eqref{eq:16} has two unique positive equilibrium point $E_1^* = \left(C_B^{1*},\,C_L^{1*}\right)$ and $E_2^* = \left(C_B^{2*},\,C_L^{2*}\right)$ with $C_L^{1*} < \widehat{C}_L < C_L^{2*}$, which are defined by \eqref{eq:17a}. Furthermore, $E_1^*$ is locally asymptotically stable, whereas, $E_2^*$ is an unstable saddle point.
\end{itemize}
\end{theorem}
{
An illustration for the conclusions of Theorem \ref{Theorem5} is given in Figure \ref{Fig:1anew}, whereas the detailed proof of Part (ii) is provided in Appendix~\ref{AppendixA}.
%%%%%%%%%%%
}
\begin{figure}[H]
\includegraphics[height=10.0cm,width=16.0cm]{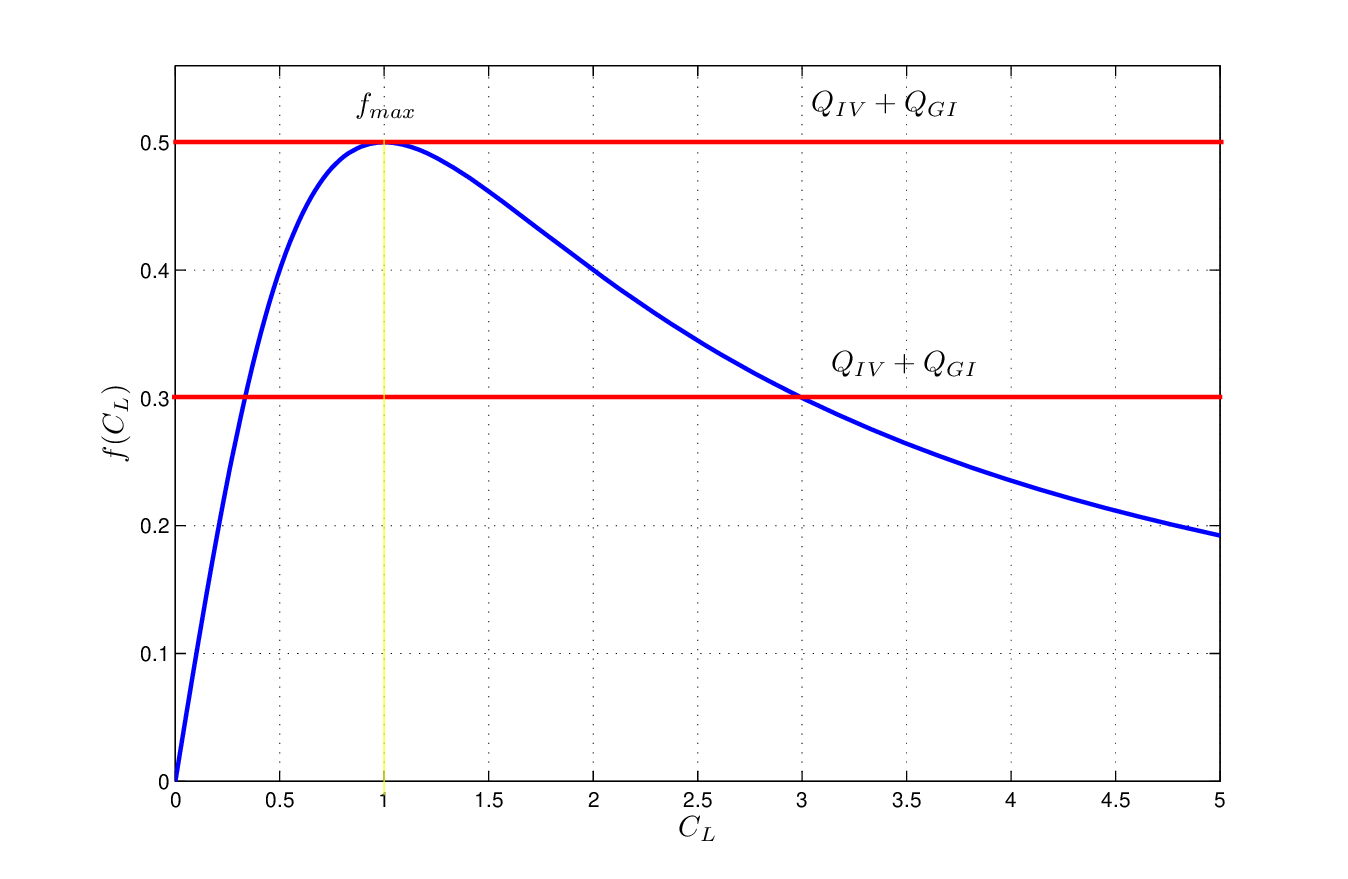}
\caption{The number of solutions to the equation $f(C_L) = Q_{IV}+ Q_{GI}$ with $f(C_L) = \frac{C_L}{1 + C_L^2}$.}\label{Fig:1anew}
\end{figure}

\begin{remark}
The dynamics of the models with other non-monotonic rate functions can be analyzed in an entirely analogous manner. It is observed from Theorem \ref{Theorem5} that the dynamics of the models with non-monotonic rate functions are more complex than those of the models with monotonic rate functions.
\end{remark}
When $Q_{IV} = Q_{GI} = 0$, the generalized model \eqref{eq:1new} is reduced to
\begin{equation}\label{eq:2new}
\begin{split}
&V_B\dfrac{dC_B}{dt} = F_{H V} \left(C_L - C_B\right),\\
&V_L\dfrac{dC_L}{dt} = F_{H V}\left(C_B - C_L\right) - f(C_L),
\end{split}
\end{equation}
%%%%
which has only a trivial equilibrium point. By using the quadratic Lyapunov function defined in \eqref{eq:12}, we conclude that this equilibrium point is globally asymptotically stable. This means that the elimination of ethanol from the body will be completed.
%%%%%%%
%%
%%
\begin{remark}
The absence of periodic orbits is established in Appendix \ref{AppendixB}.
\end{remark}
\section{Numerical simulations}\label{Sec4}
\subsection{Numerical simulations with Michaelis--Menten function}
Here, we investigate the dynamical behaviour of the model \eqref{eq:1} with the Michaelis--Menten mechanism. For this purpose, we consider it with the parameters given in Table \ref{Table1}.
\begin{table}[H]
\caption{The parameters for the model \eqref{eq:1}.}\label{Table1}
\centering
\begin{tabular}{ccccccccccccccccc}
\hline
Set & $F_{HV}$ & $V_B$ & $V_L$ & $V_{max}$ & $K_m$ & $Q_{IV}$ & $Q_{GI}$ & Source & $E^*$\\
\hline
1 & 1.5 & 48 & 0.61 & 2.75 & 0.1 & $\frac{100}{60}$ & 0 & \cite{Levitt} & $(1.2650,\,0.1538)$\\
%%%%%%%%%%%%%%%%%%%%%%%%%%%%%%%%%%
\hline
2 & 1.5 & 48 & 0.61 & 2.75 & 0.1 & 0 & $\frac{100}{60}$ & \cite{Levitt} & $(0.1538,\,0.1538)$\\
\hline
3 & 1.5 & 48 & 0.61 & 2.75 & 0.1 & $\frac{200}{60}$ & 0 & \cite{Levitt} & Does not exist\\
\hline
4 & 1.5 & 48 & 0.61 & 2.75 & 0.1 & 0 & $\frac{200}{60}$ & \cite{Levitt} & Does not exist\\
\hline
\end{tabular}
\end{table}
We consider ethanol kinetics under the following two scenarios:
%%%%%%%%%%%%%%%%%%%%%%%%%%%%%%%%%%%%%%%%%%%%%%%%%%%%%
\begin{enumerate}
    \item During continuous ethanol administration via either the intravenous route ($Q_{IV}>0$) or the gastrointestinal route ($Q_{GI}>0$). In this case, the ethanol kinetics are governed by \eqref{eq:1}.

    \item After ethanol administration has ceased, i.e., $Q_{IV}=Q_{GI}=0$. In this case, the ethanol kinetics are governed by \eqref{eq:1S}, with no external ethanol input.
\end{enumerate}

In the numerical results reported below, we use the classical four-stage Runge-Kutta method (RK4) \cite{Ascher, Stuart} with a step size of $10^{-4}$ to obtain approximate solutions. The obtained numerical solutions are depicted in Figures \ref{Fig:1}--\ref{Fig:4}.

We observe from Figures \ref{Fig:1} and \ref{Fig:2} that the unique positive equilibrium points are globally asymptotically stable, whereas, Figure \ref{Fig:3} indicates the solutions are increasing over time.

{
Figure~\ref{Fig:4} illustrates a simplified scenario of ethanol consumption and subsequent elimination,
using Parameter Set~2 in Table~\ref{Table1}. We assume that
$100$ mmol of ethanol is absorbed through the
gastrointestinal tract at a constant rate over
$60$ minutes. Specifically, $Q_{IV}=0$ and
\[
Q_{GI}(t)=
\begin{cases}
\dfrac{100}{60}, & 0\leq t\leq 60,\\
0, & t>60.
\end{cases}
\]
After absorption ceases, ethanol is gradually
metabolized and eliminated, and both concentrations
$C_B(t)$ and $C_L(t)$ converge to zero.
This behavior is consistent with the physiological
process of ethanol elimination after alcohol
consumption and illustrates the theoretical
results on the global asymptotic stability
of the ethanol-free equilibrium.
}

%Figure \ref{Fig:4} illustrates the solution of the model when ethanol is administered via the gastrointestinal tract for 60 minutes with parameter Set $1$ in Table \ref{Table1}. After the administration is stopped, ethanol is gradually eliminated by the body, and its concentration decreases to zero. This behavior is fully consistent with physiological reality. Consequently, the theoretical assertions presented in Section \ref{Sec2} are supported and illustrated.

\begin{figure}[H]
\subfloat[$C_B(0) = C_L(0) = 0$]{%
\includegraphics[height=10.5cm,width=9cm]{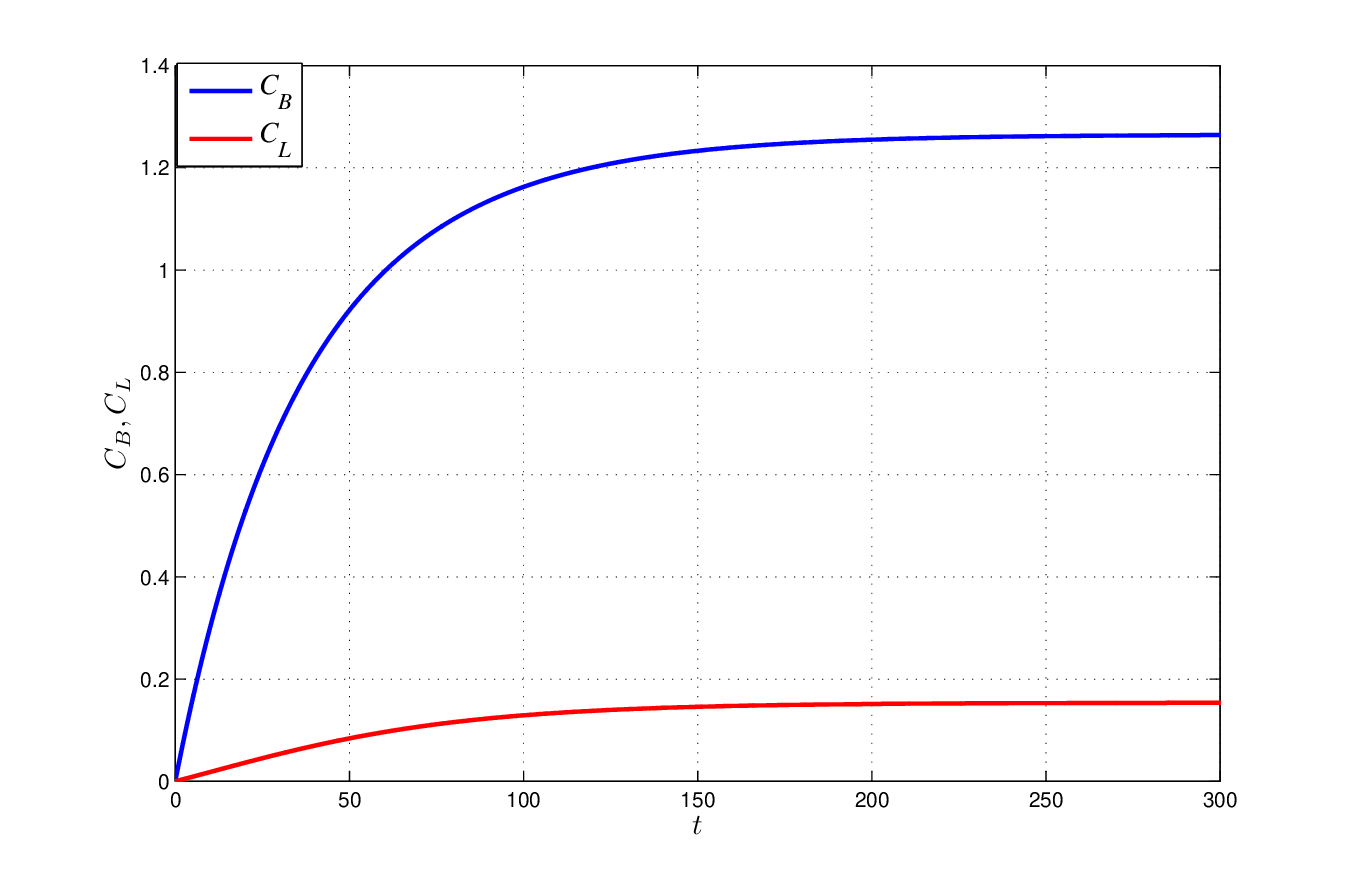}
\label{Figure:1a}
}\hfill
\subfloat[$C_B(0) = 0.5$,\,\,\,$C_L(0) = 0.25$]{%
\includegraphics[height=10.5cm,width=9cm]{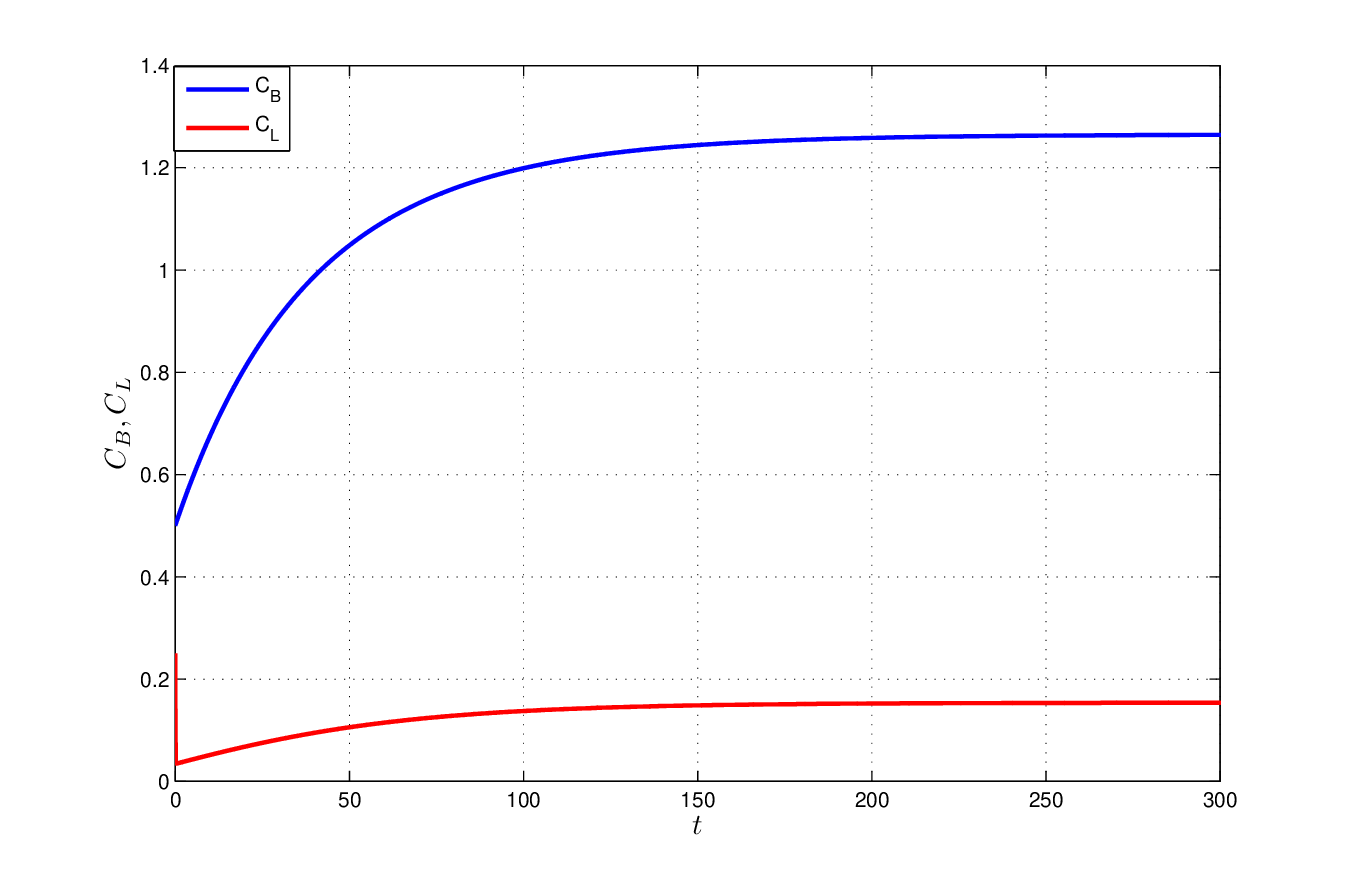}
\label{Figure:1b}
}\hfill
%%%
%%%
\subfloat[$C_B(0) = 1.0$,\,\,\, $C_L(0) = 0.75$]{%
\includegraphics[height=10.5cm,width=9cm]{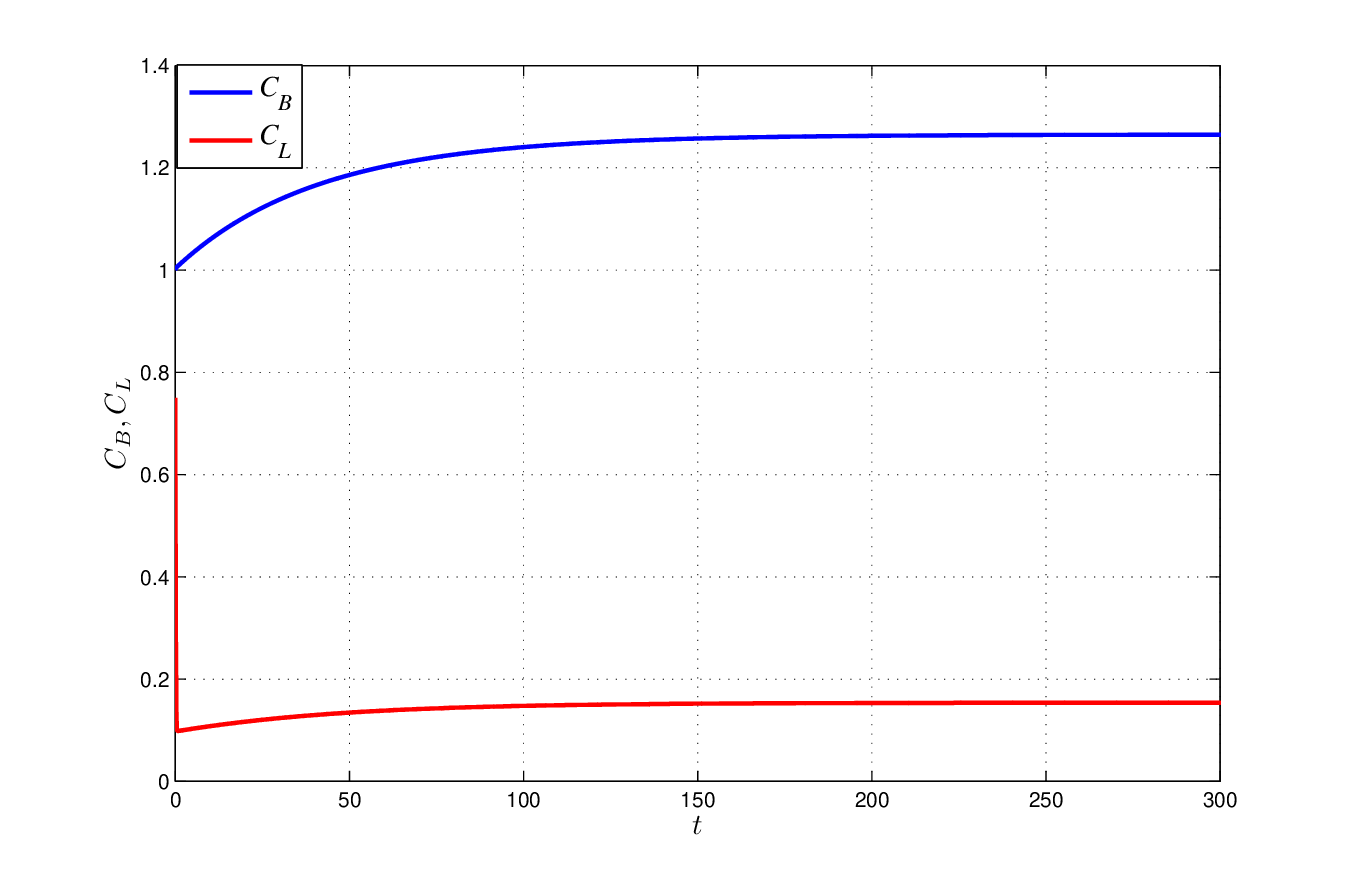}
\label{Figure:1c}
}\hfill
\subfloat[$C_B(0) = 2.0$,\,\,\,$C_L(0) = 1.5$]{%
\includegraphics[height=10.5cm,width=9cm]{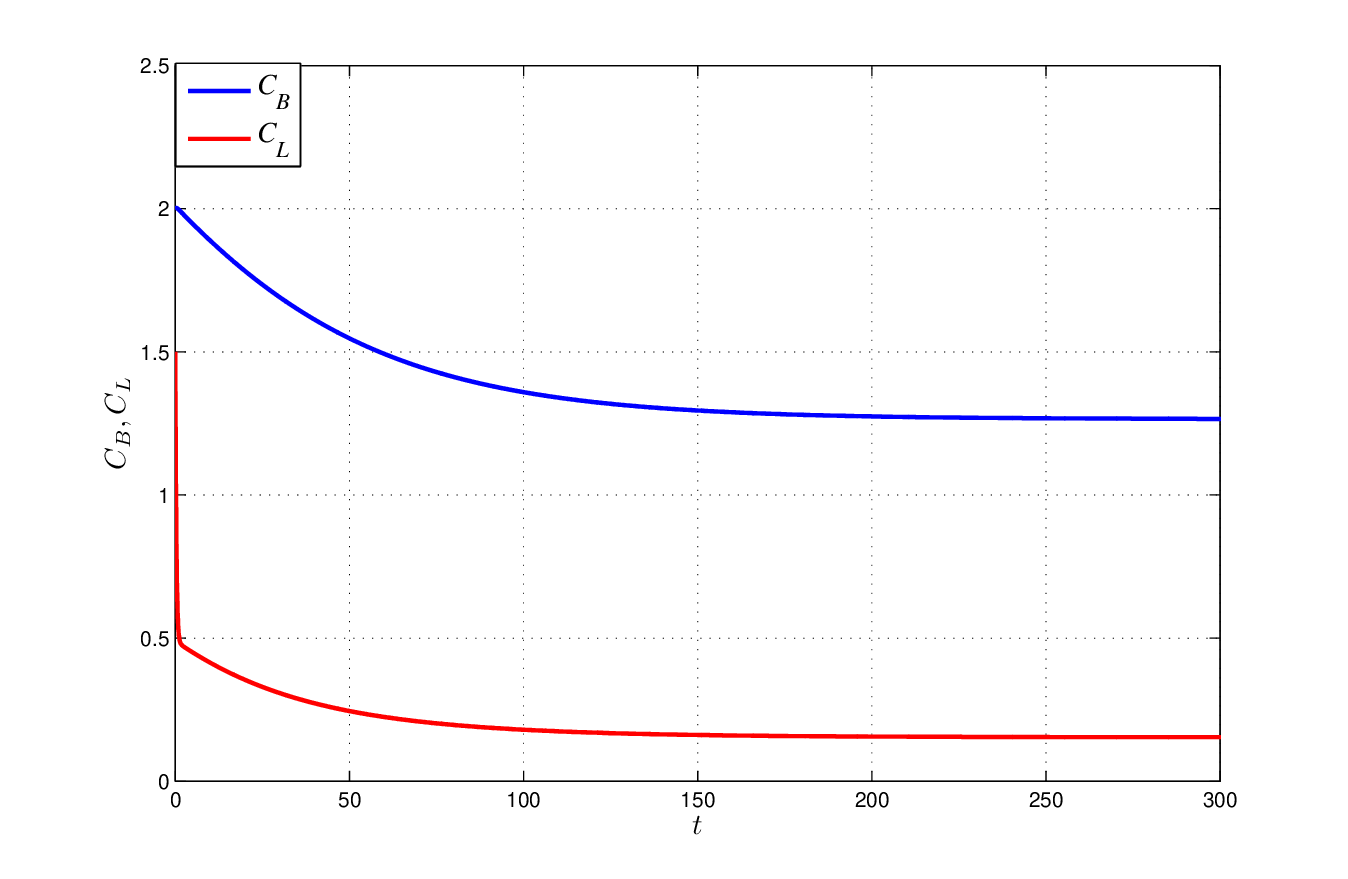}
\label{Figure:1d}
}
\caption{The solutions of the model \eqref{eq:1} using parameter Set 1 in Table \ref{Table1}.}
\label{Fig:1}
\end{figure}

\begin{figure}[H]
\subfloat[$C_B(0) = C_L(0) = 0$]{%
\includegraphics[height=10.5cm,width=9cm]{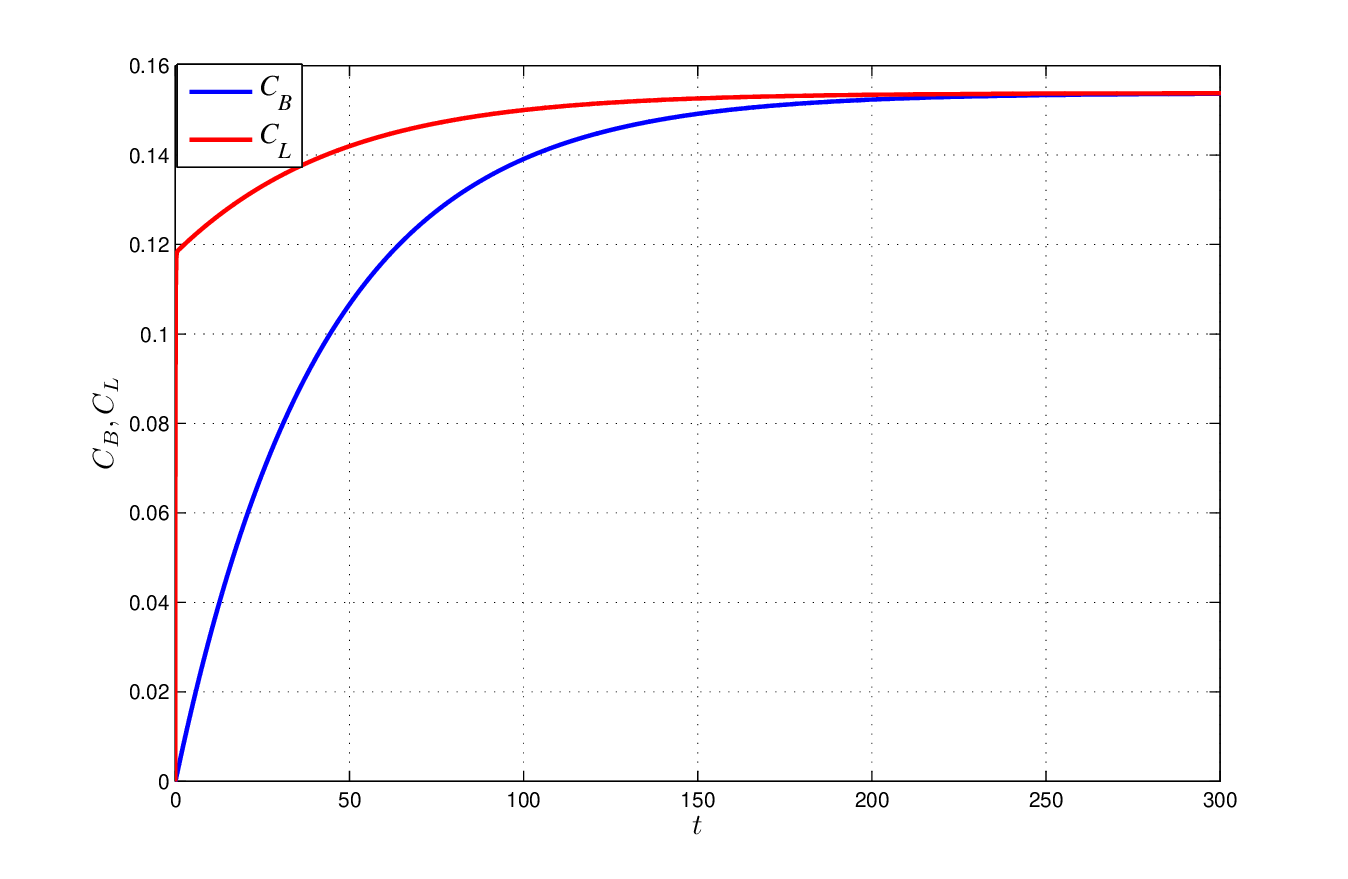}
\label{Figure:1a}
}\hfill
\subfloat[$C_B(0) = 0.5$,\,\,\,$C_L(0) = 0.25$]{%
\includegraphics[height=10.5cm,width=9cm]{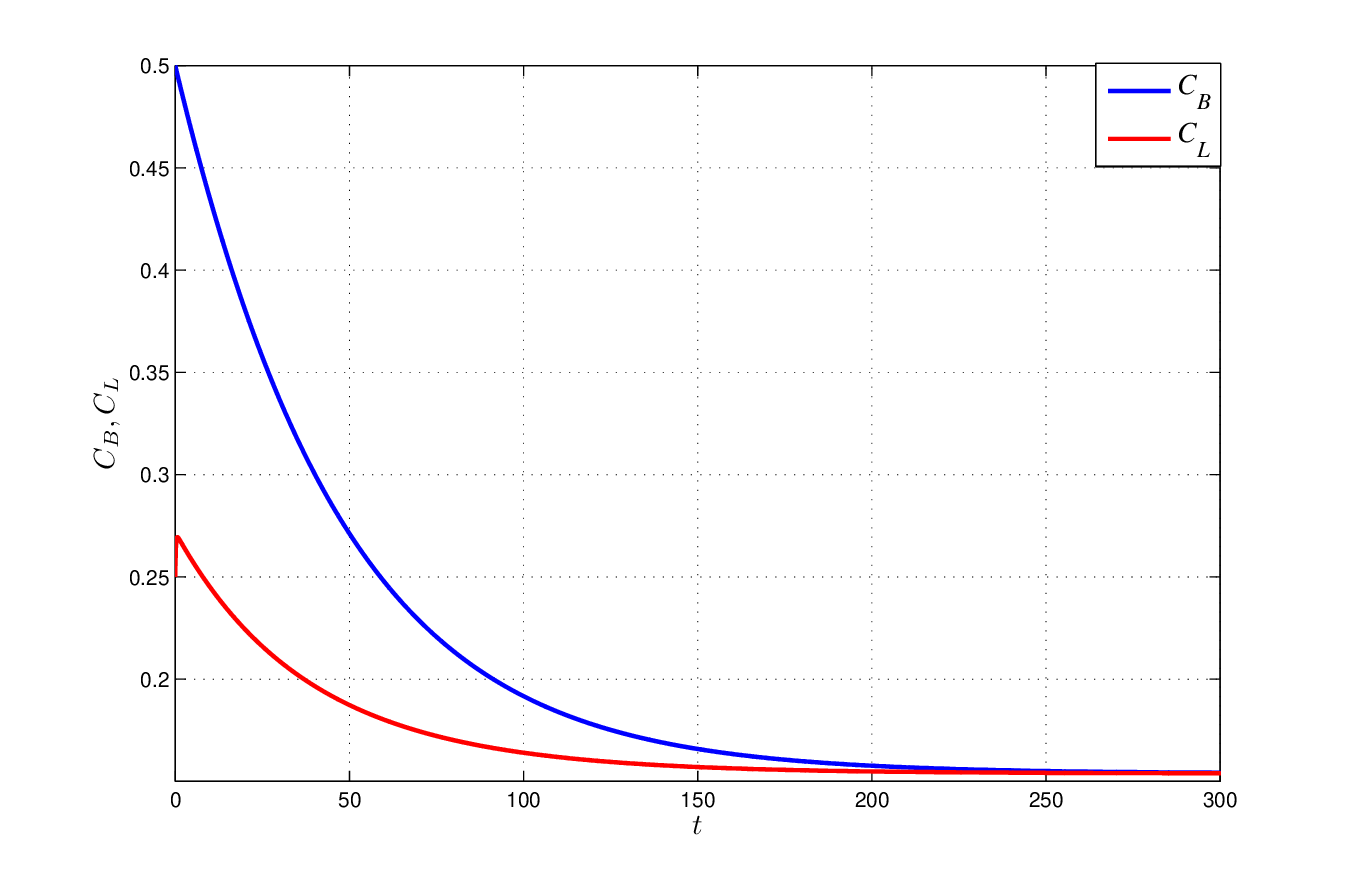}
\label{Figure:1b}
}\hfill
%%%
%%%
\subfloat[$C_B(0) = 1.0$,\,\,\, $C_L(0) = 0.75$]{%
\includegraphics[height=10.5cm,width=9cm]{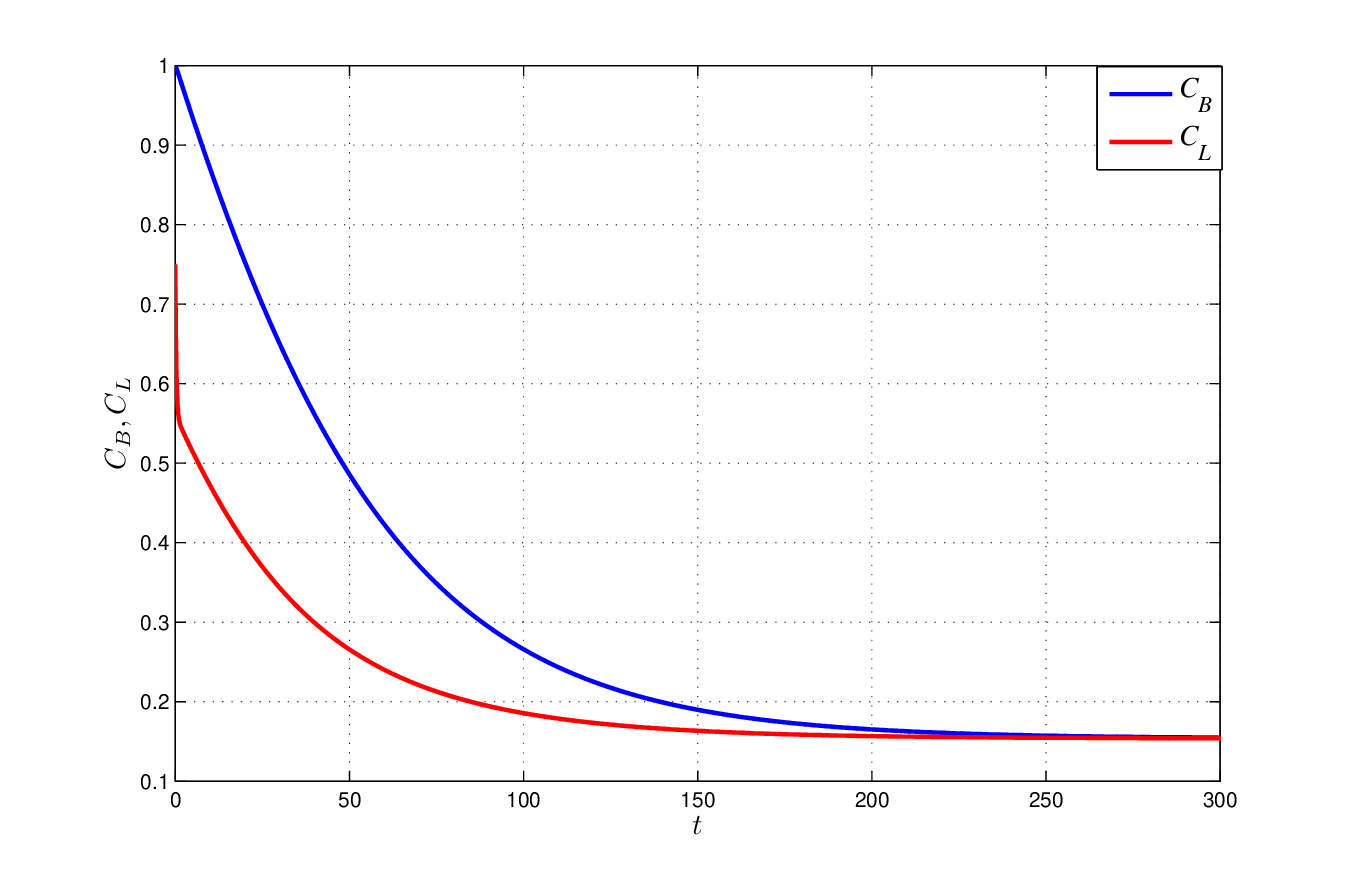}
\label{Figure:1c}
}\hfill
\subfloat[$C_B(0) = 2.0$,\,\,\,$C_L(0) = 1.5$]{%
\includegraphics[height=10.5cm,width=9cm]{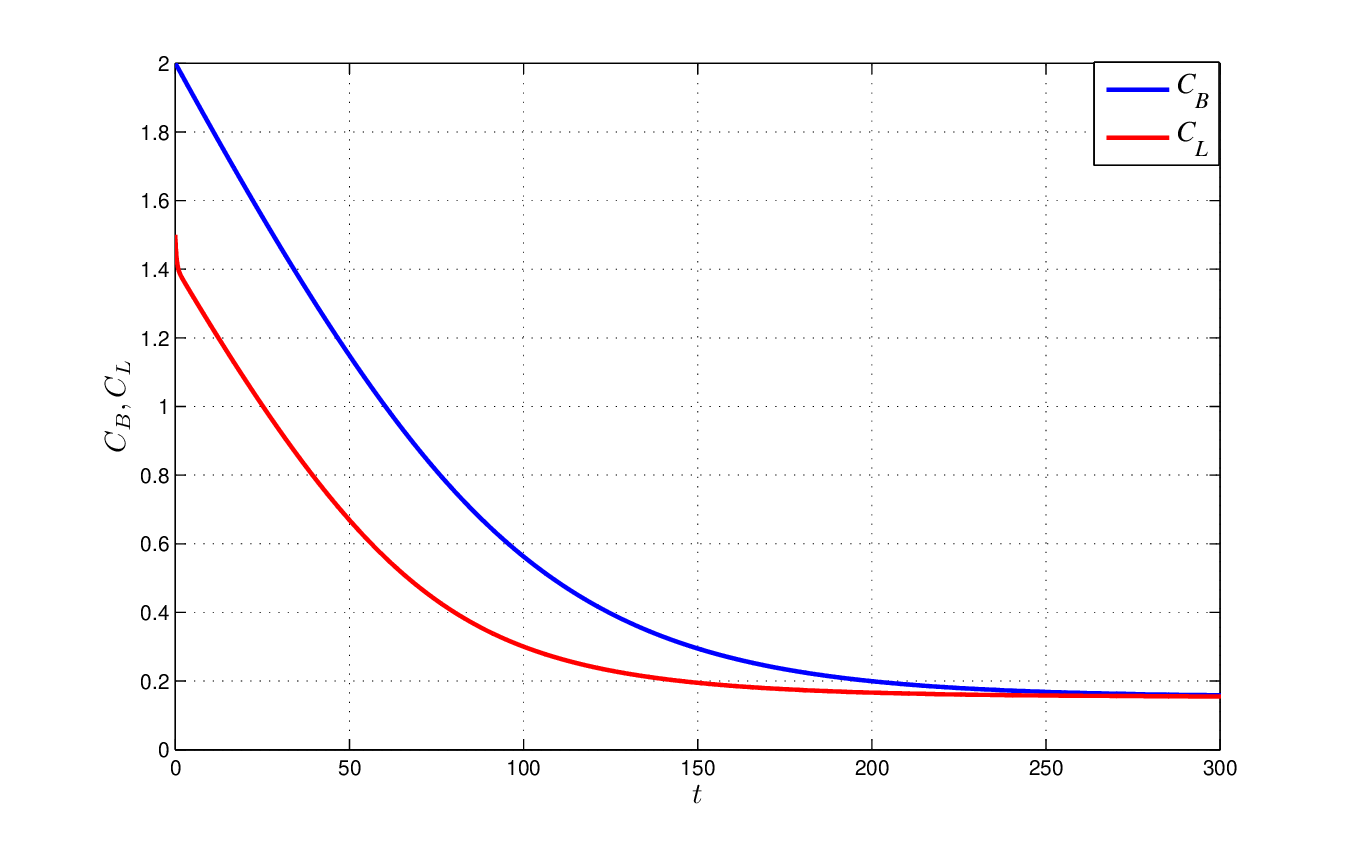}
\label{Figure:1d}
}
\caption{The solutions of the model \eqref{eq:1} using parameter Set 2 in Table \ref{Table1}.}
\label{Fig:2}
\end{figure}

\begin{figure}[H]
\subfloat[Parameter Set $3$]{%
\includegraphics[height=9.5cm,width=9cm]{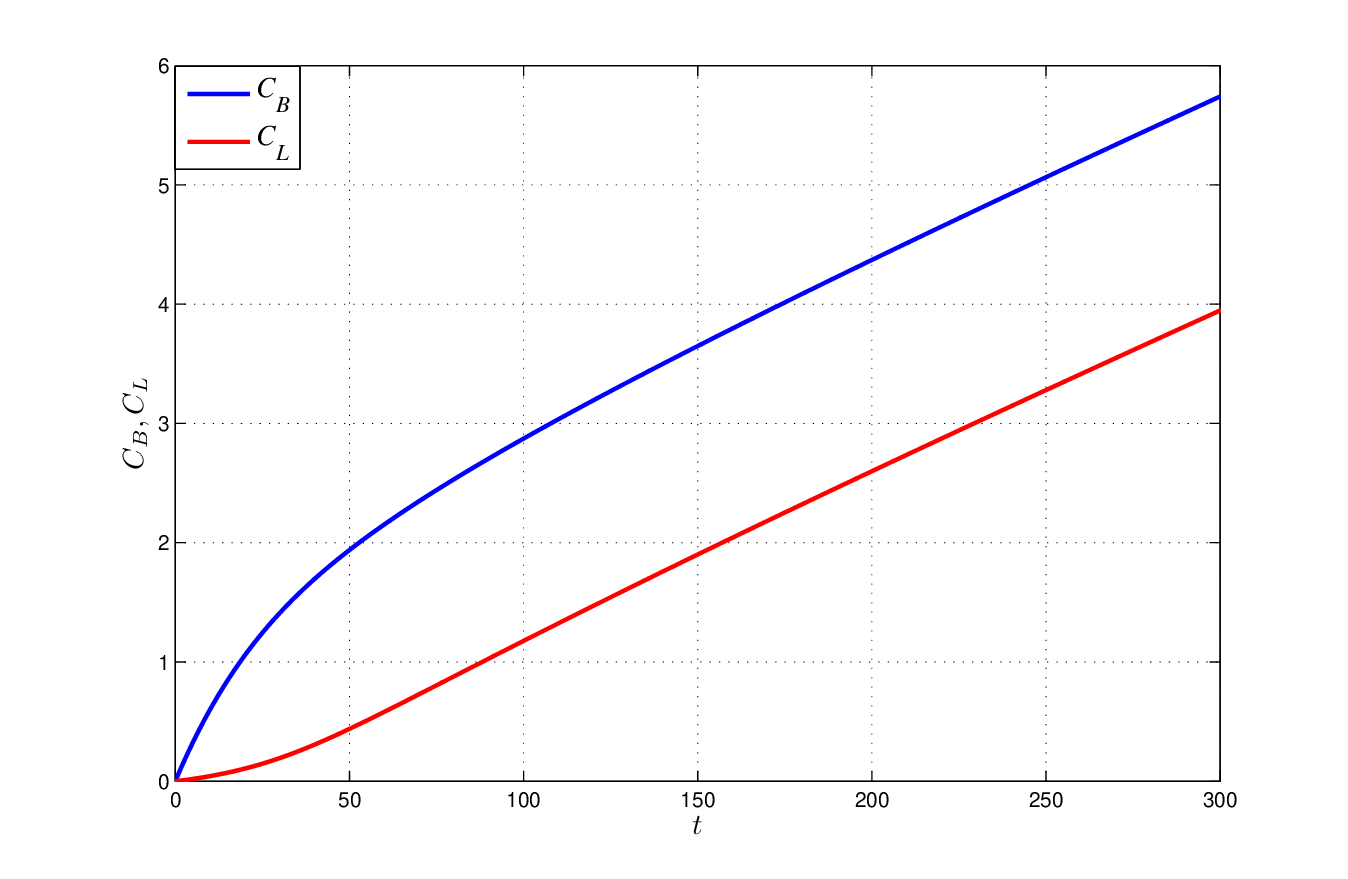}
\label{Figure:1a}
}\hfill
\subfloat[Parameter Set $4$]{%
\includegraphics[height=9.5cm,width=9cm]{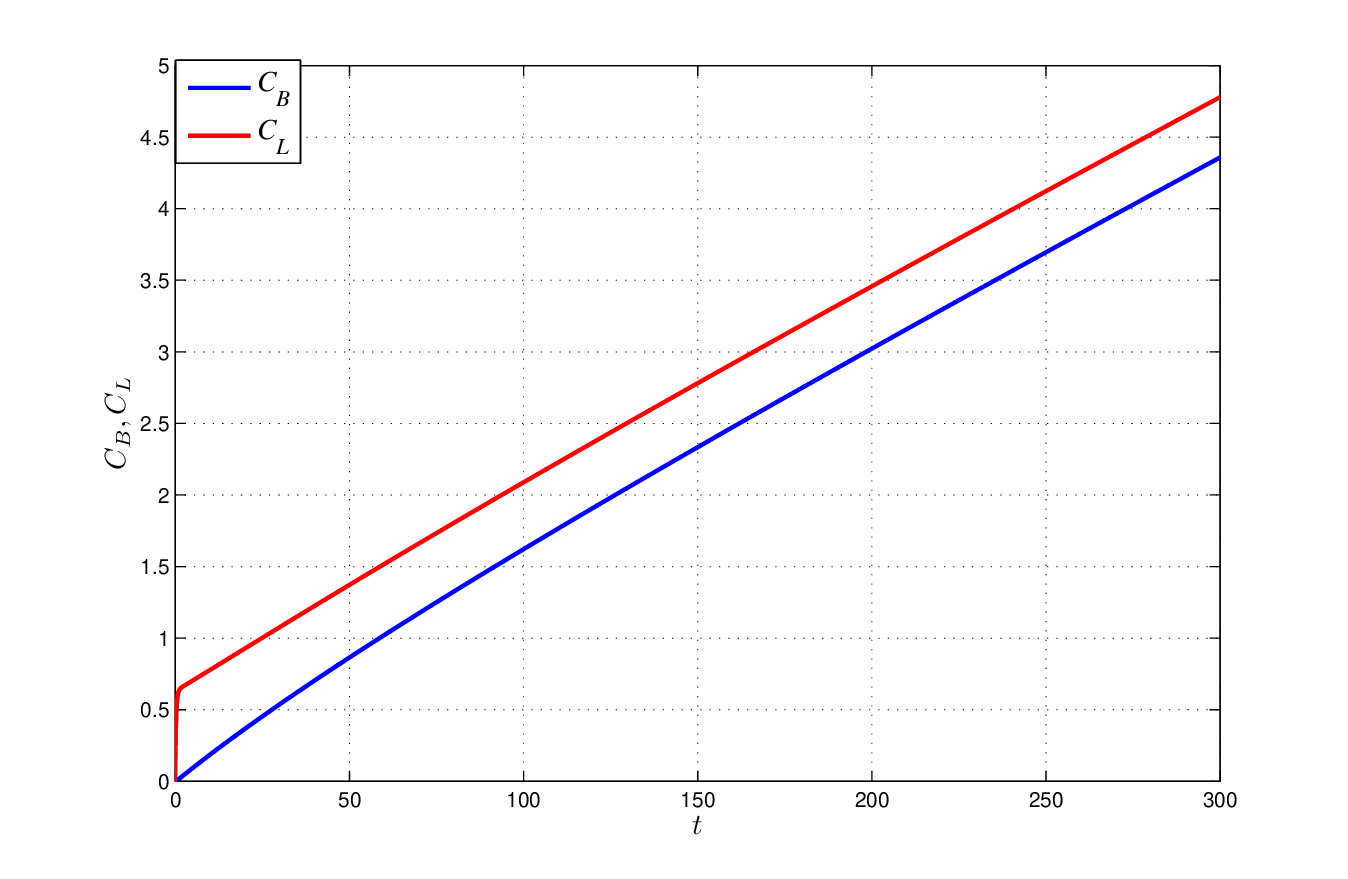}
\label{Figure:1b}
}
\caption{The solutions of the model \eqref{eq:1} with parameter Sets 3 and 4 in Table \ref{Table1}.}
\label{Fig:3}
\end{figure}
\begin{figure}[H]
\subfloat[$C_B$-component]{%
\includegraphics[height=9.5cm,width=9cm]{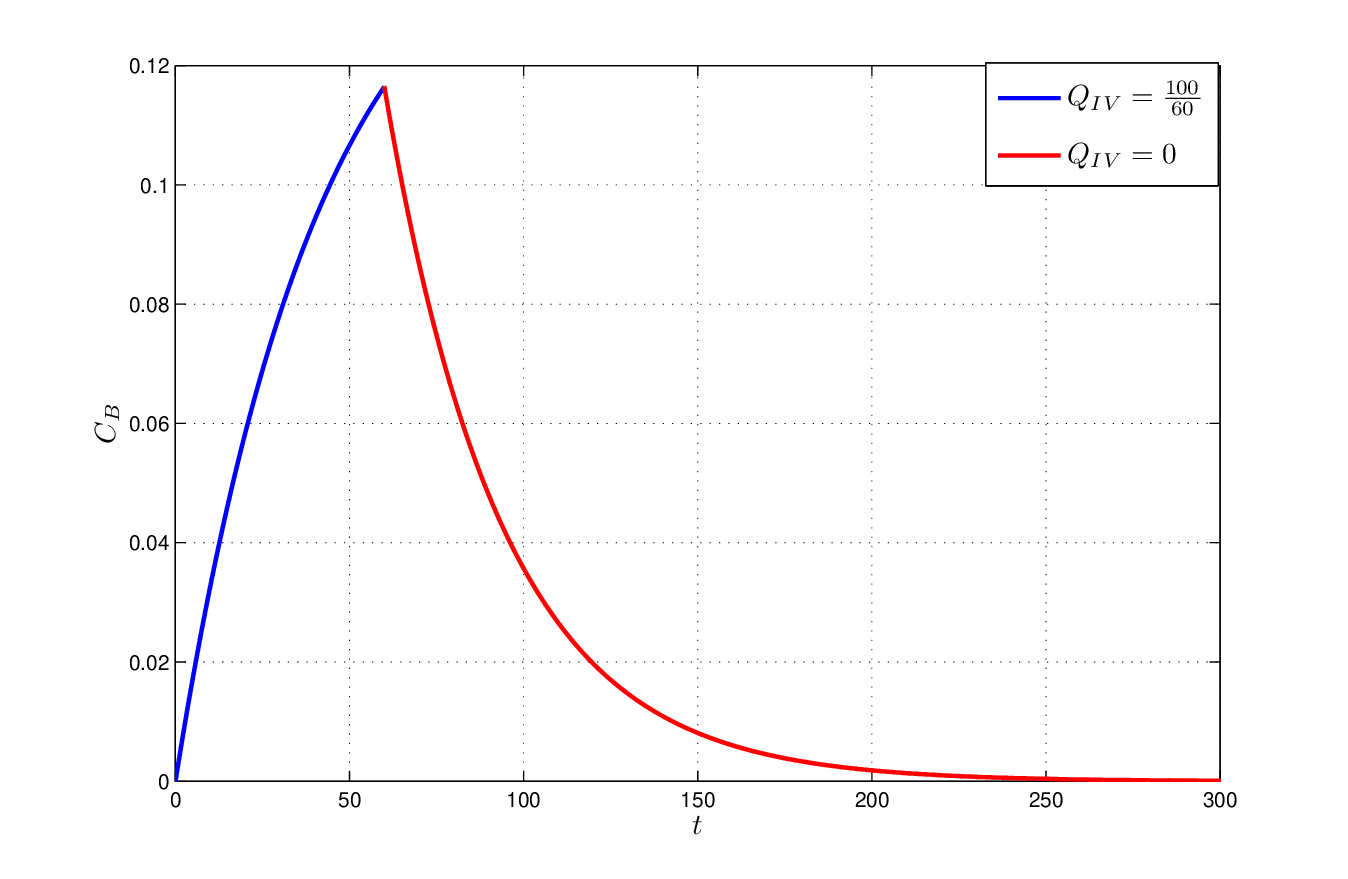}
\label{Figure:1a}
}\hfill
\subfloat[$C_L$-component]{%
\includegraphics[height=9.5cm,width=9cm]{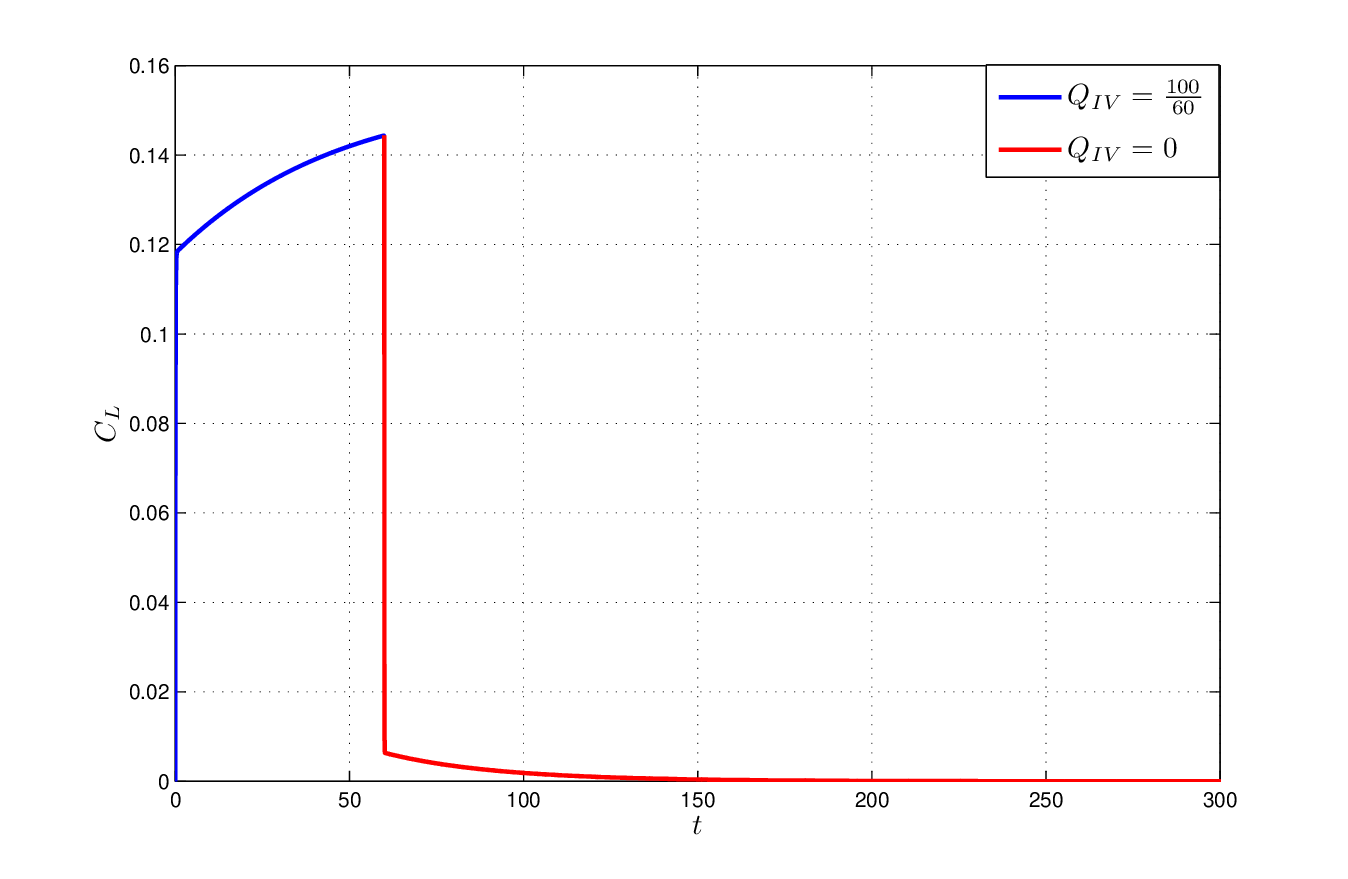}
\label{Figure:1b}
}
\caption{The solution of the model \eqref{eq:1} using parameter Set $1$ in Table \ref{Table1} with $Q_{IV} = \frac{100}{60}$ for $ t \leq 60$ and $Q_{IV} = 0$ for $t > 60$.}
\label{Fig:4}
\end{figure}
\subsection{Numerical simulations with a nonmonotone rate function}
In this subsection, we examine the dynamics of the model \eqref{eq:1new} with a nonmonotone rate function defined by
\begin{equation*}
f(C_L) = \dfrac{\kappa_1 C_L}{\kappa_2 + C_L^2}, \quad \kappa_1, \kappa_2 >  0,
\end{equation*}
which describes the psychological (inhibitory or overload) effect. Consequently, the model under consideration is given by
\begin{equation}\label{eq:25}
\begin{split}
&V_B\dfrac{dC_B}{dt} = F_{H V} \left(C_L - C_B\right)+ Q_{I V},\\
&V_L\dfrac{dC_L}{dt} = F_{H V}\left(C_B - C_L\right) + Q_{G I}- \dfrac{\kappa_1 C_L}{\kappa_2 + C_L^2}.
\end{split}
\end{equation}
In the following numerical experiments, we consider \eqref{eq:25} with the parameters given in Table \ref{Table2}.

The solutions of the model, generated by employing the RK4 method with a step size of $10^{-4}$, are represented in Figures \ref{Fig:5}--\ref{Fig:8}. From these figures, we see that depending on the initial conditions, the solutions either converge to the positive equilibrium or exhibit unbounded growth over time. Therefore, the numerical results are consistent with and support the local asymptotic stability analysis presented in Section \ref{Sec3}.

\begin{table}[H]
\caption{The parameters for the generalized model \eqref{eq:25}.}\label{Table2}
\centering
\begin{tabular}{ccccccccccccccccc}
\hline
Set & $F_{HV}$ & $V_B$ & $V_L$ & Source & $Q_{GI}$ & $Q_{IV}$ & $\kappa_1$ & $\kappa_2$ & Source & Equilibrium point\\
\hline
1 & 1.5 & 48 & 0.61 & \cite{Levitt} & 2.0 & 0 &  4.0 & 1.0 & Assumed & $(1.0000,\,1.0000)$\\
%%%%%%%%%%%%%%%%%%%%%%%%%%%%%%%%%%
\hline
2 & 1.5 & 48 & 0.61 & \cite{Levitt} & 1.0 & 0 & 4.0 & 1.0 & Assumed & $(0.2679,\,0.2679)$ (stable)\\
&&&&&&&&&& $(3.7321,\,3.7321)$ (unstable)\\
\hline
3 & 1.5 & 48 & 0.61 & \cite{Levitt} & 0 & 2.0 & 4.0 & 1.0 &  Assumed & {$(2.3333,\,1.0000)$}\\
\hline
4 & 1.5 & 48 & 0.61 & \cite{Levitt} & 0 & 1.0 & 4.0 & 1.0 & Assumed & {$(0.9346,\,0.2679)$ (stable)}\\
&&&&&&&&&& {$(4.3987,\,3.7321)$ (unstable)}\\
%4 & 1.5 & 48 & 0.61 & \cite{Levitt} & 0 & 1.0 & 4.0 & 1.0 & Assumed & $(0.9346,\,0.2679)$ (stable)\\
%%
%&&&&&&&&&& $(4.3987,\,3.7321)$ (unstable)\\
%4 & 1.5 & 48 & 0.61 & \cite{Levitt} & 0 & 1.0 & 4.0 & 1.0 & Assumed & $(0.2679,\,1.6013)$ (stable)\\
%%%
%&&&&&&&&&& $(3.7321,\,5.0654)$ (unstable)\\
\hline
\end{tabular}
\end{table}
\begin{figure}[H]
\subfloat[$C_B(0) = C_L(0) = 0.00$]{%
\includegraphics[height=10.5cm,width=9cm]{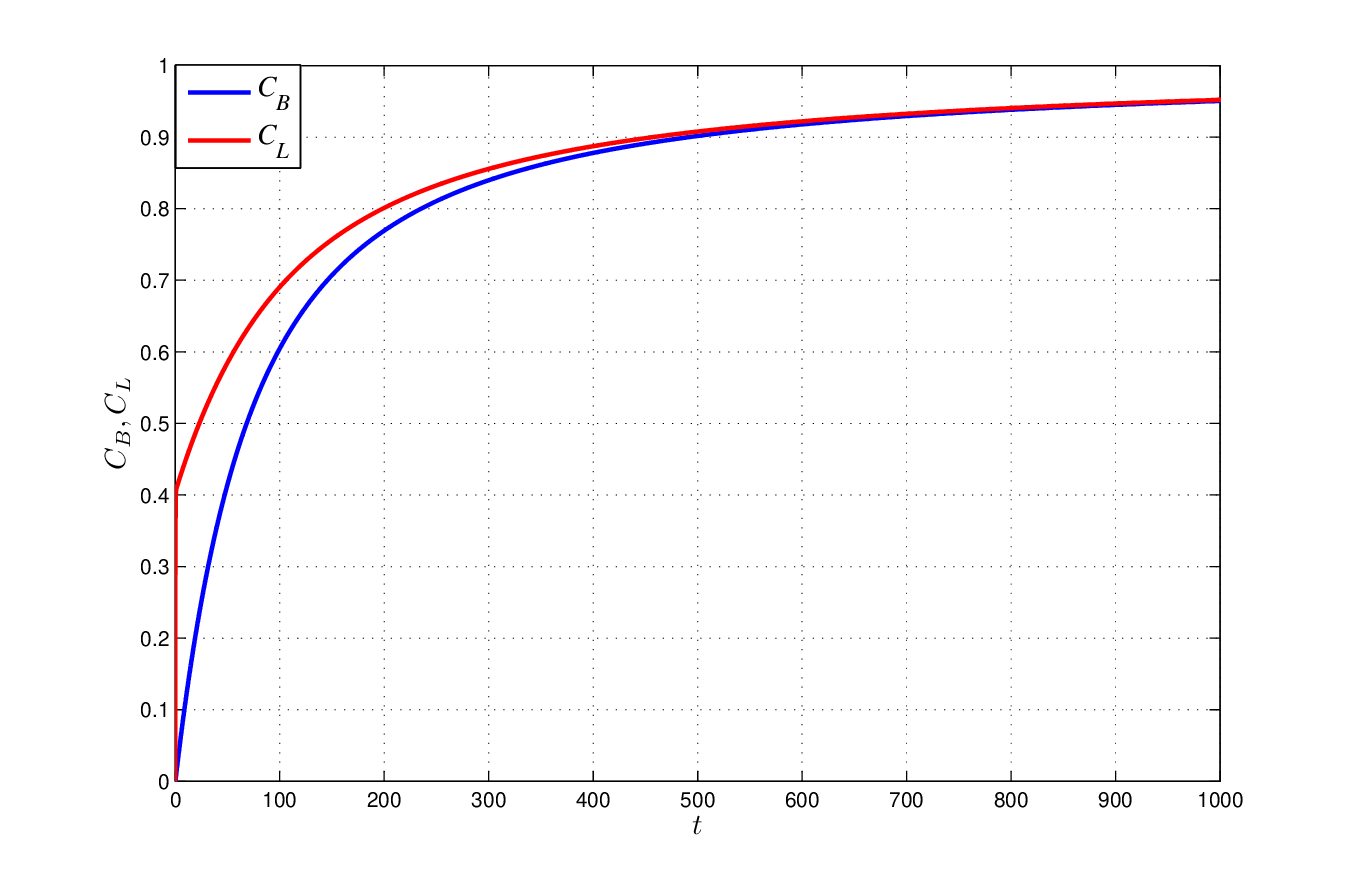}
\label{Figure:1a1}
}\hfill
\subfloat[$C_B(0) = 0.5$,\,\,\,$C_L(0) = 0.25$]{%
\includegraphics[height=10.5cm,width=9cm]{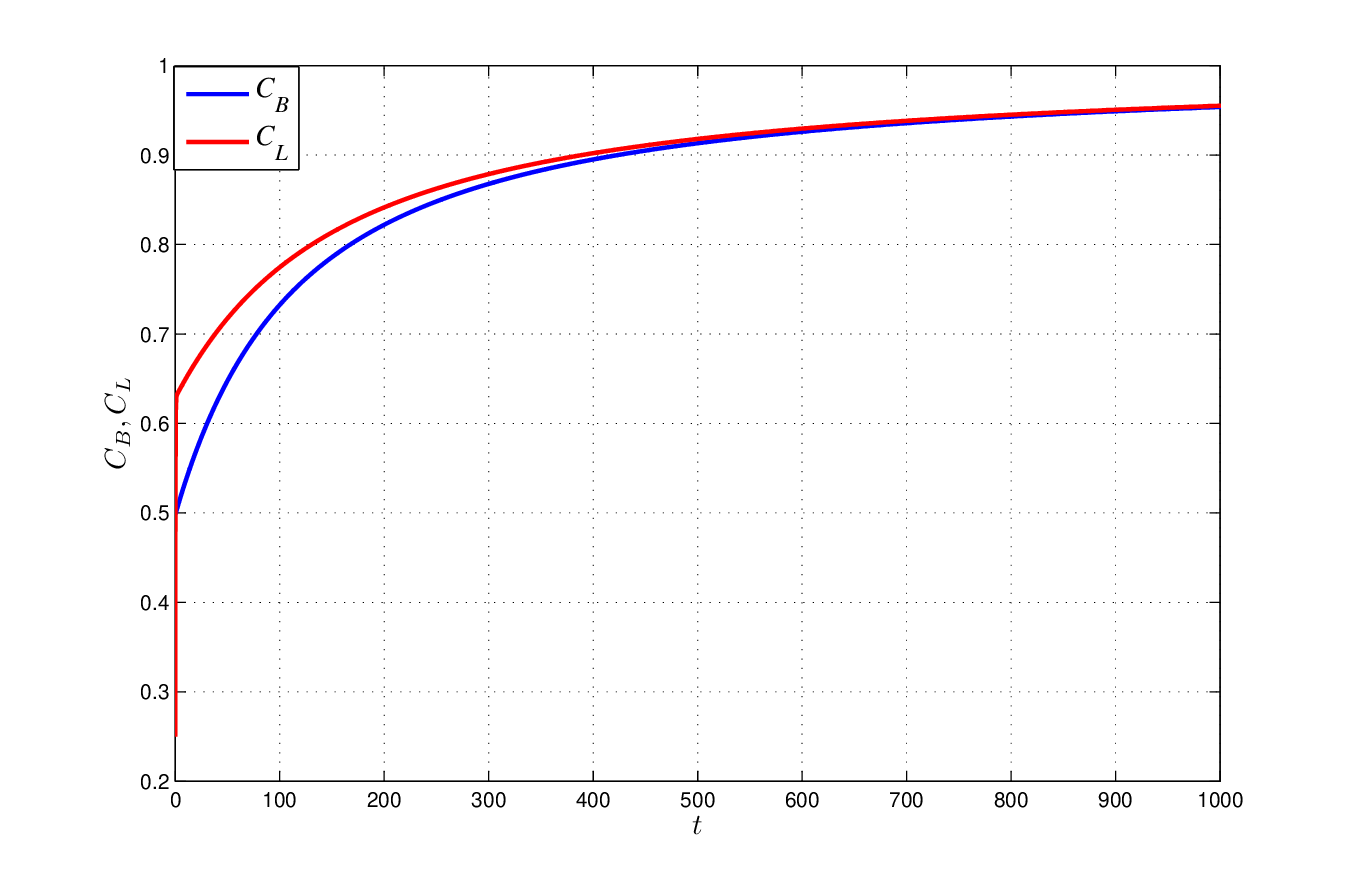}
\label{Figure:1b1}
}\hfill
%%%
%%%
\subfloat[$C_B(0) = 1.50$,\,\,\, $C_L(0) = 1.50$]{%
\includegraphics[height=10.5cm,width=9cm]{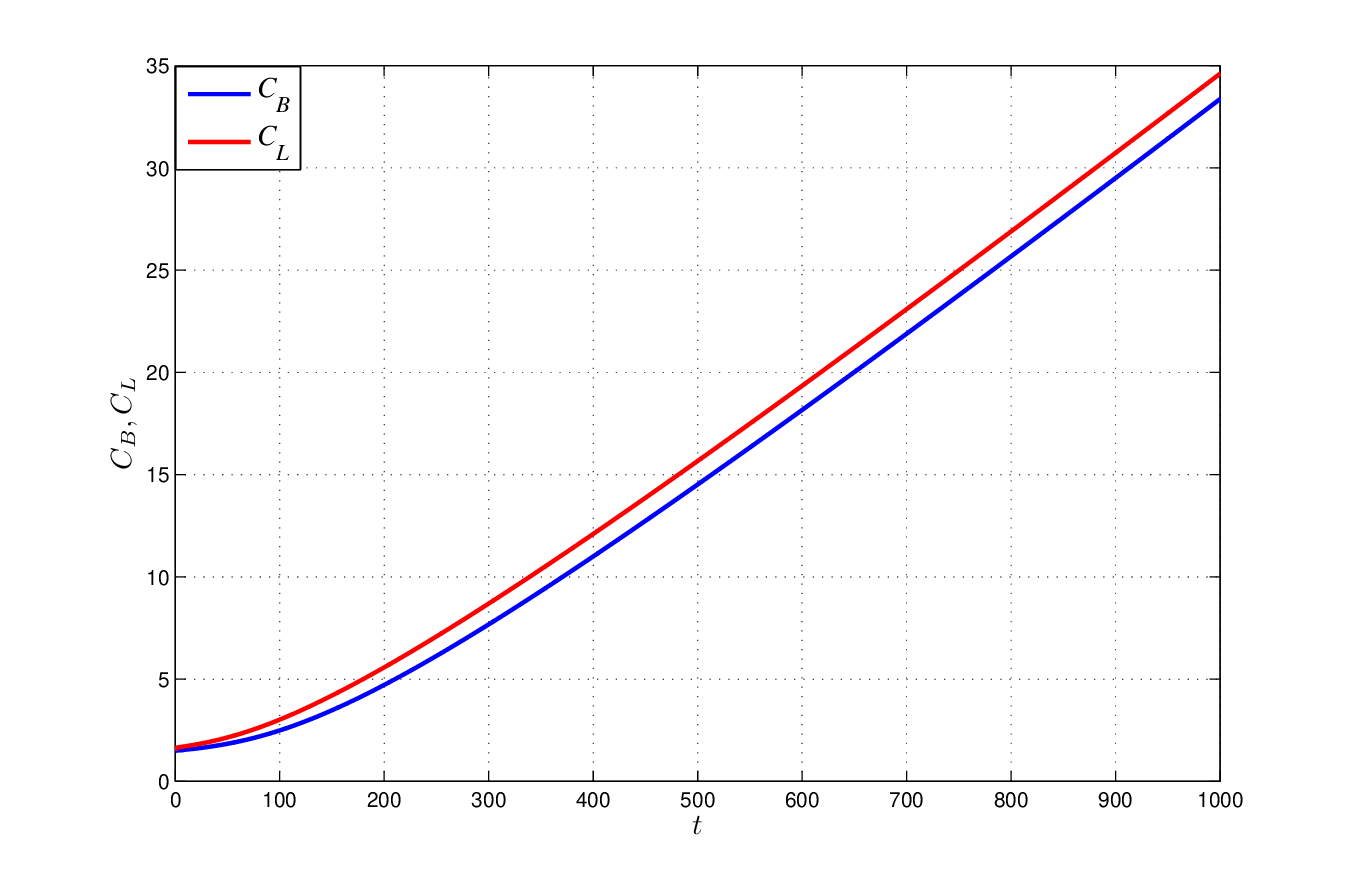}
\label{Figure:1c1}
}\hfill
\subfloat[{$C_B(0) = 2.00$,\,\,\,$C_L(0) = 2.00$}]{%
\includegraphics[height=10.5cm,width=9cm]{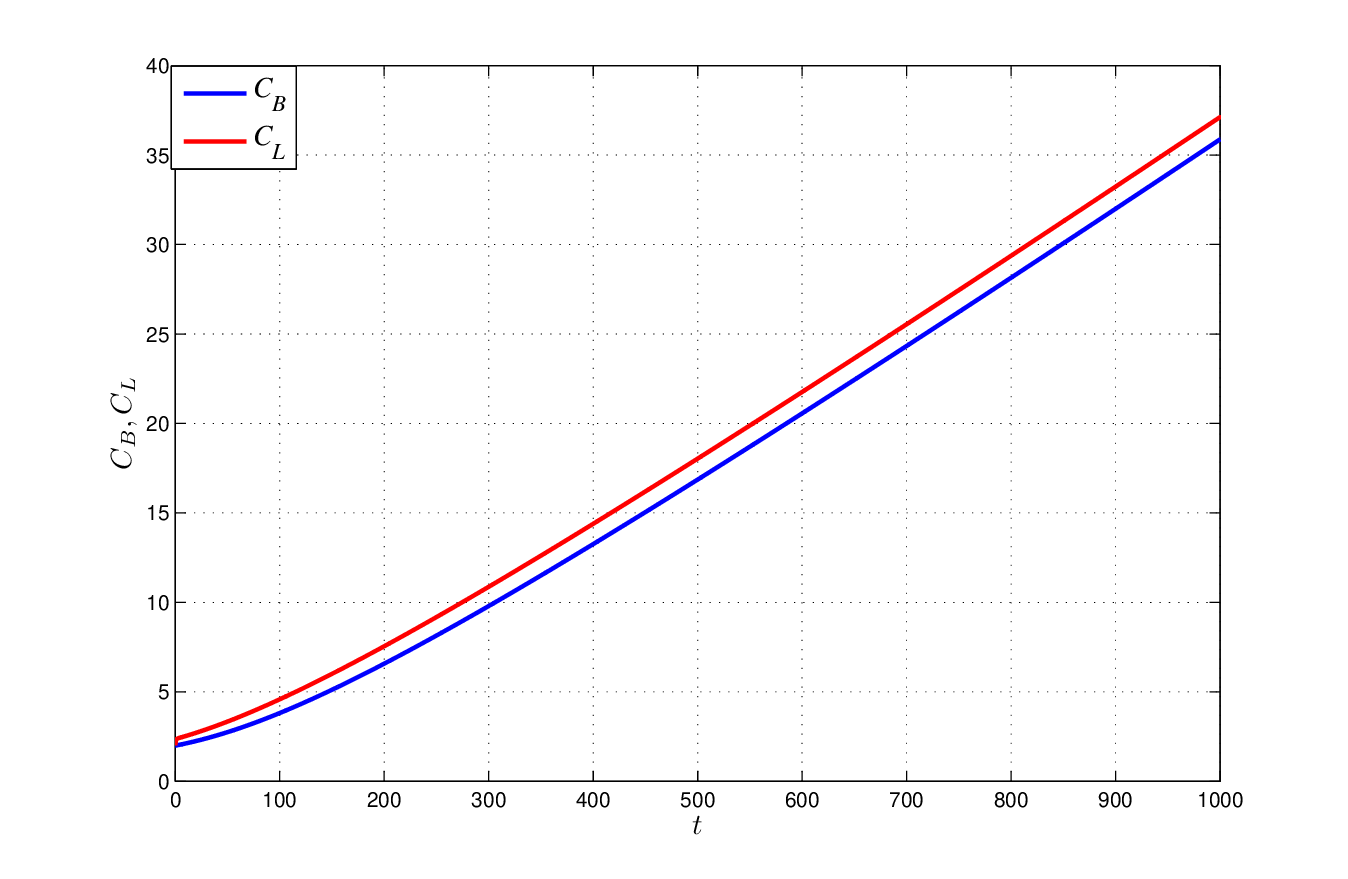}
\label{Figure:1d1}
}
\caption{The solutions of the model \eqref{eq:25} using parameter Set $1$ in Table \ref{Table2}.}
\label{Fig:5}
\end{figure}
%%%%%%%%%
%%%%%%%%%%%%%%%%%%
\begin{figure}[H]
\subfloat[$C_B(0) = C_L(0) = 0.00$]{%
\includegraphics[height=10.5cm,width=9cm]{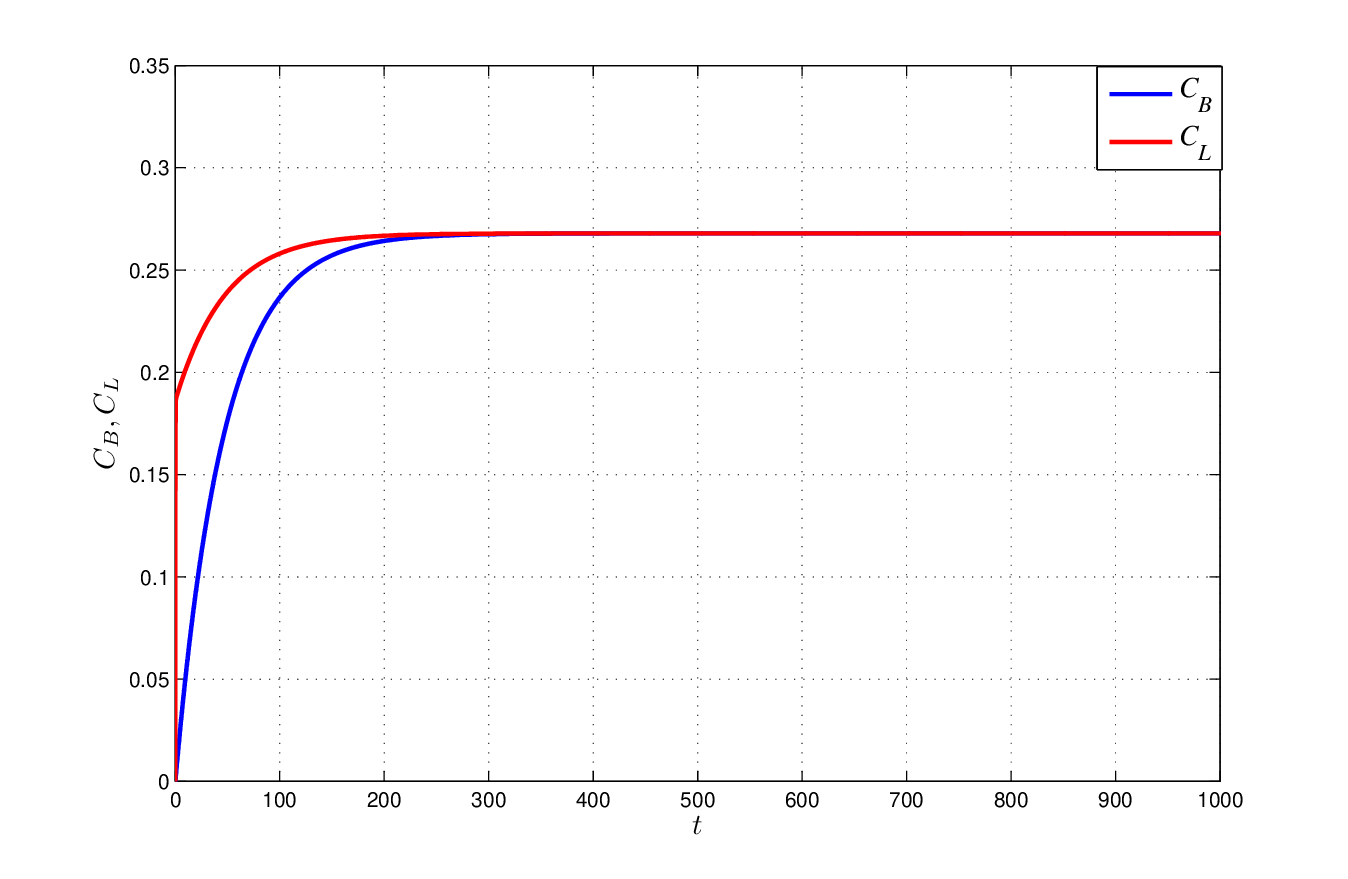}
\label{Figure:1a2}
}\hfill
\subfloat[$C_B(0) = 0.50$,\,\,\,$C_L(0) = 0.25$]{%
\includegraphics[height=10.5cm,width=9cm]{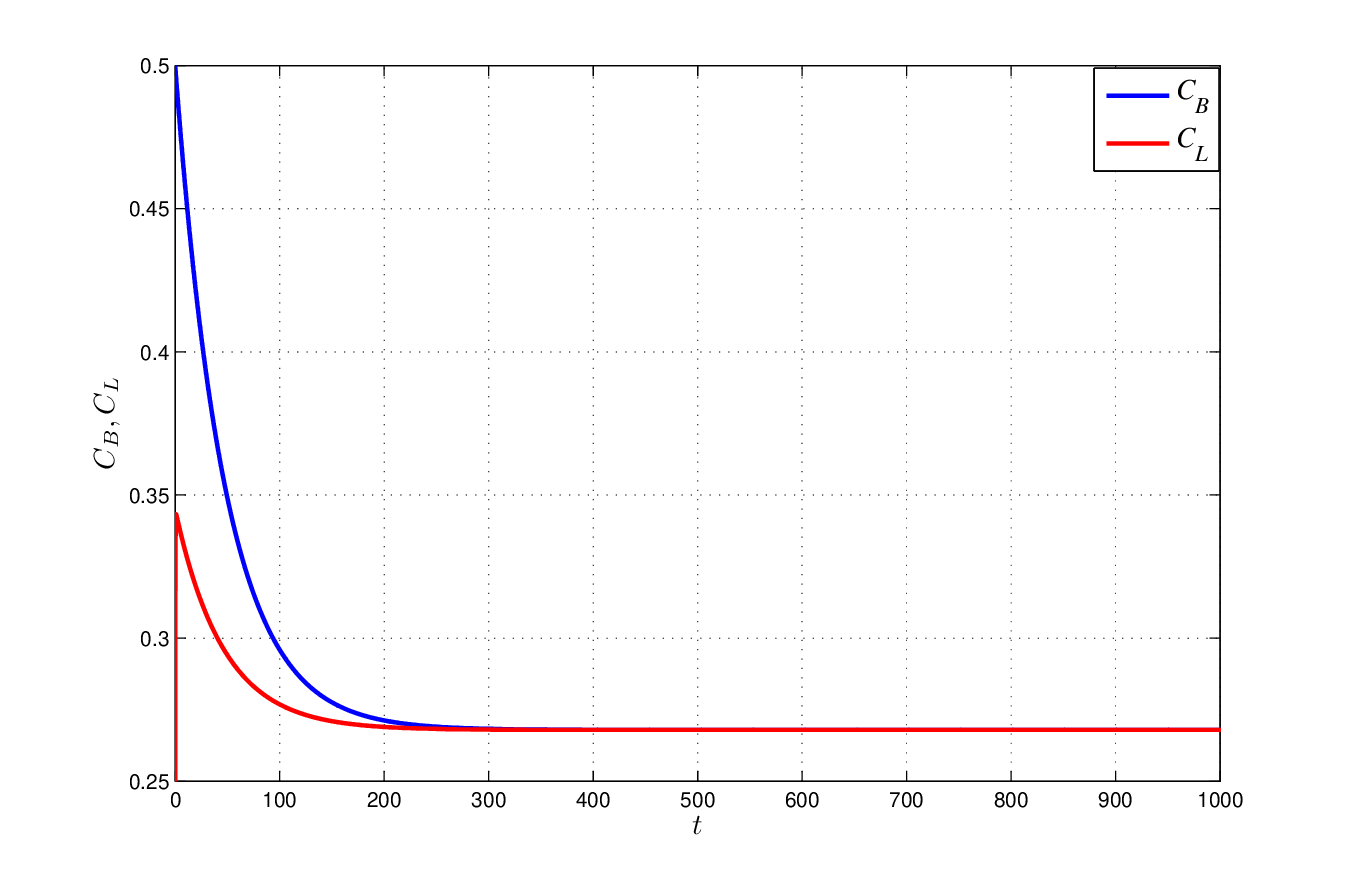}
\label{Figure:1b2}
}\hfill
%%%
%%%
\subfloat[$C_B(0) = 1.50$,\,\,\, $C_L(0) = 1.50$]{%
\includegraphics[height=10.5cm,width=9cm]{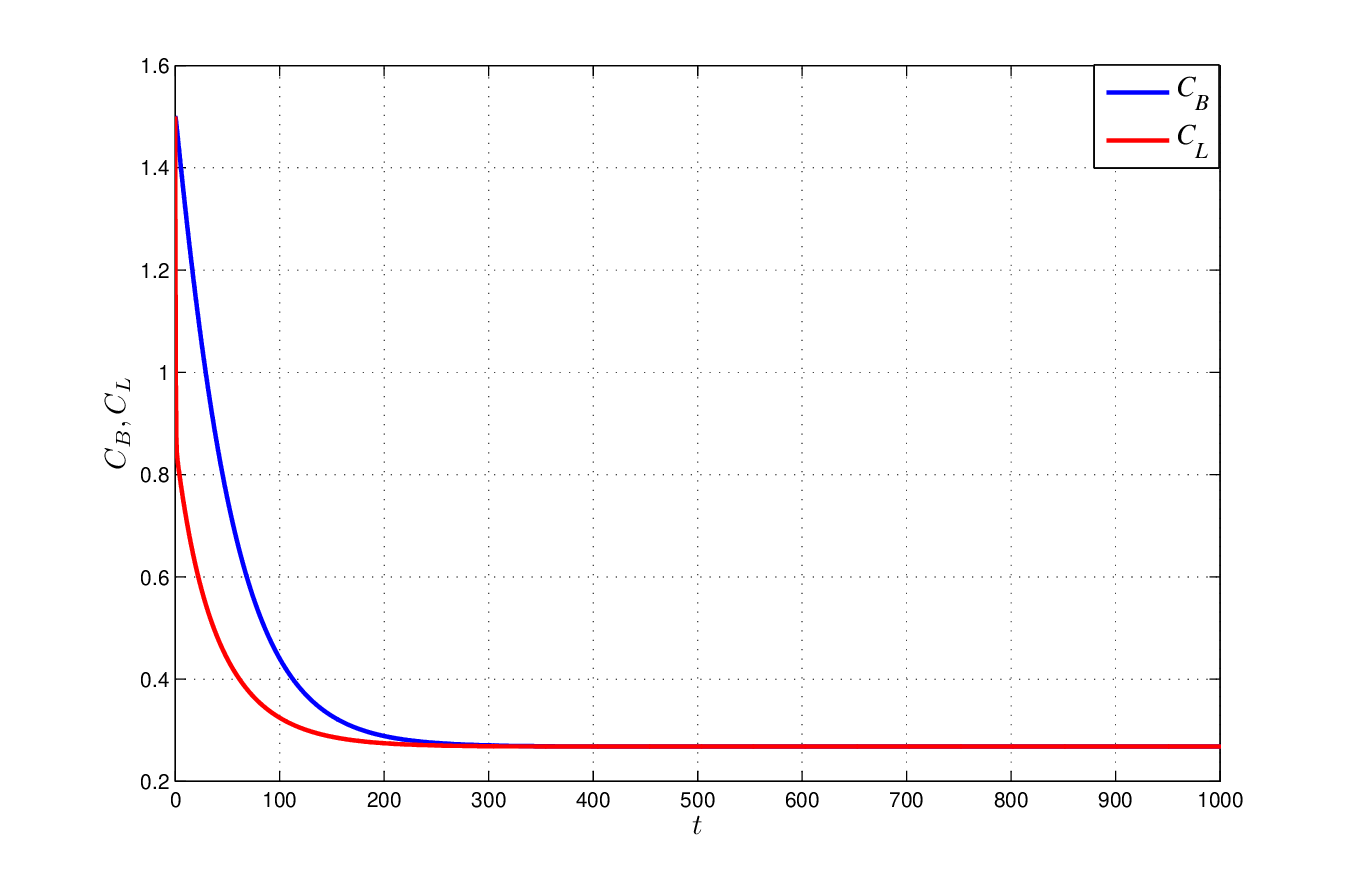}
\label{Figure:1c2}
}\hfill
\subfloat[$C_B(0) = 5.00$,\,\,\,$C_L(0) = 3.00$]{%
\includegraphics[height=10.5cm,width=9cm]{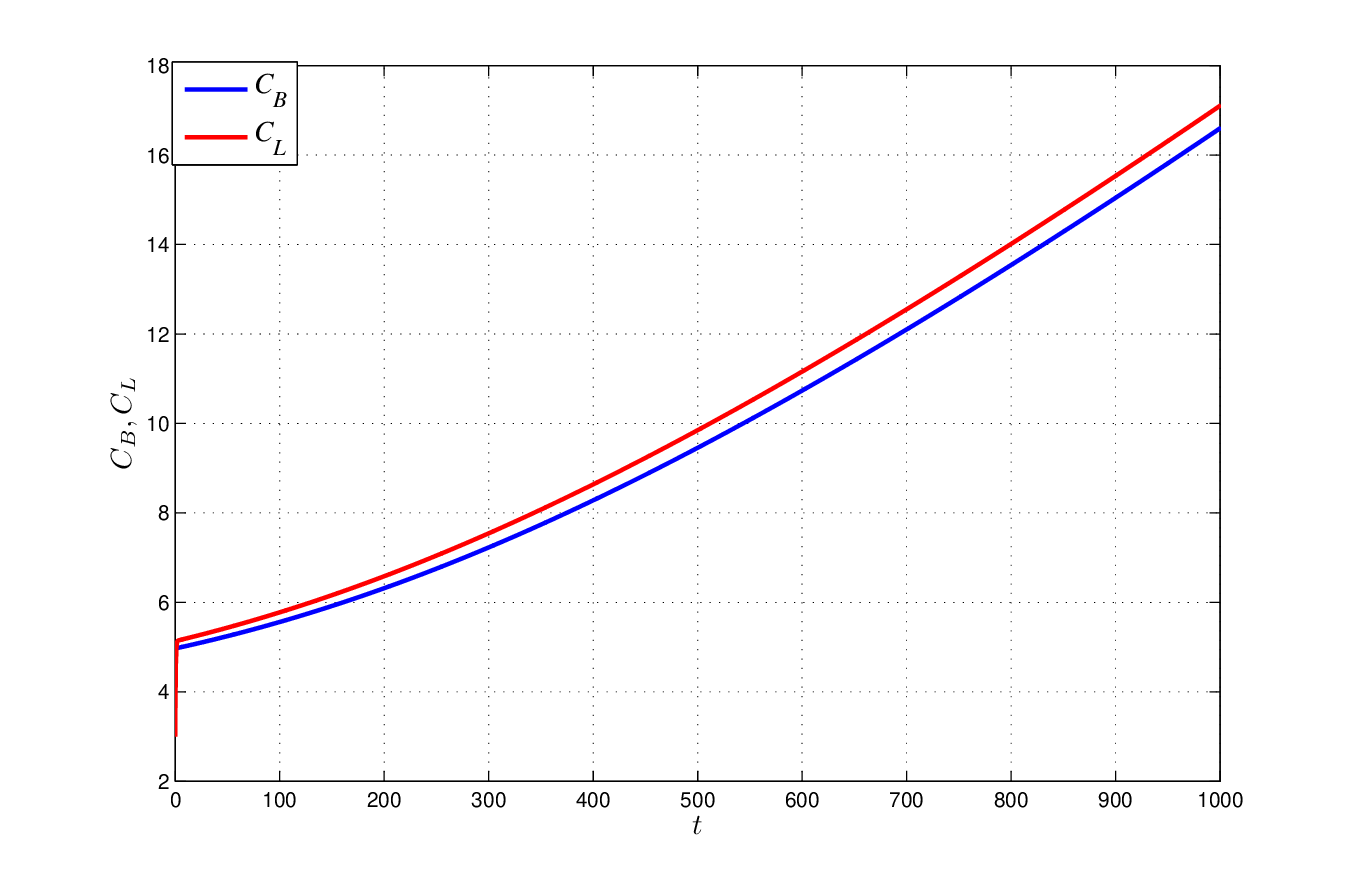}
\label{Figure:1d2}
}
\caption{The solutions of the model \eqref{eq:25} using parameter Set $2$ in Table \ref{Table2}.}
\label{Fig:6}
\end{figure}
\begin{figure}[H]
\subfloat[$C_B(0) = C_L(0) = 0.00$]{%
\includegraphics[height=10.5cm,width=9cm]{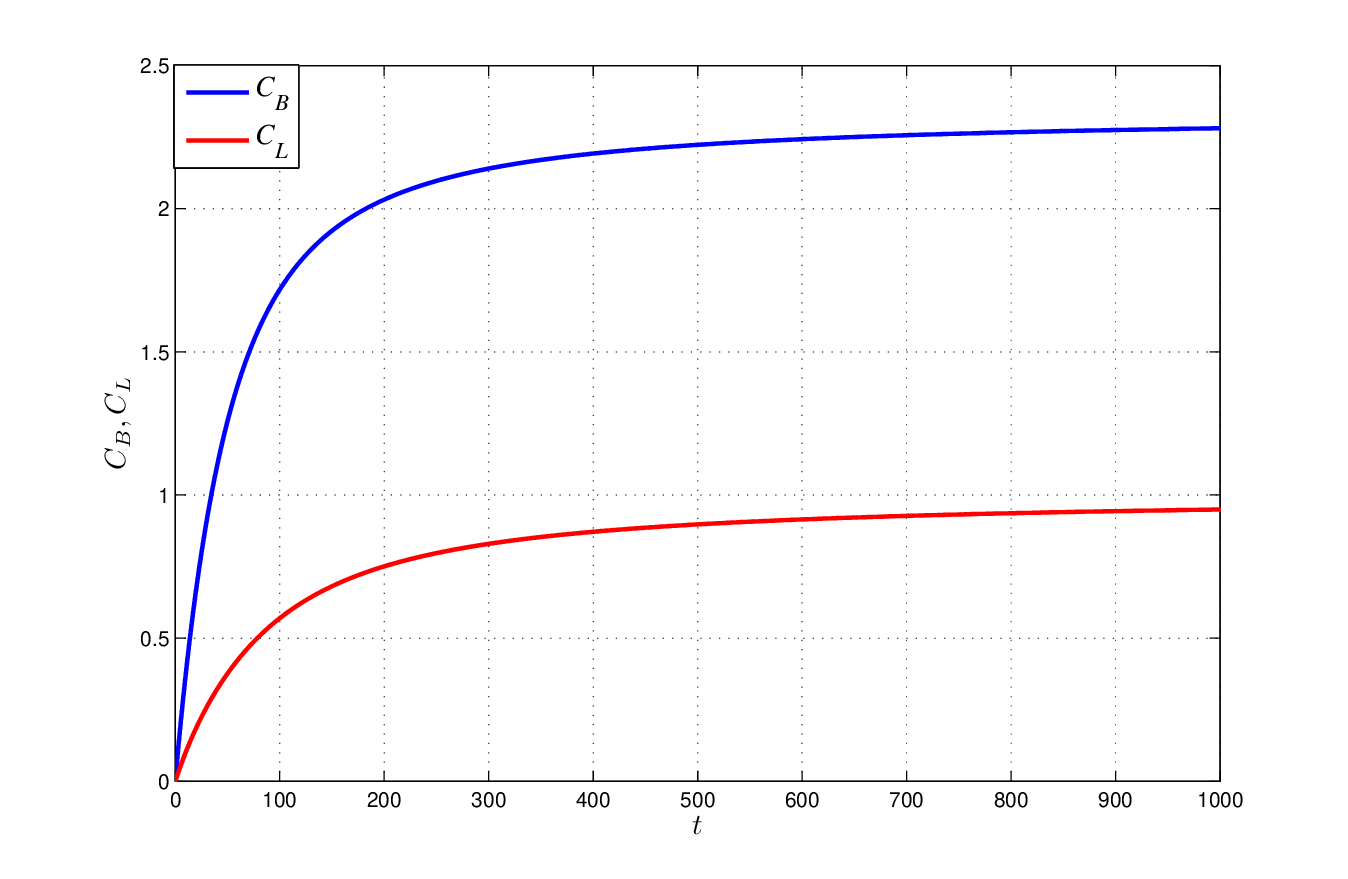}
\label{Figure:1a3}
}\hfill
\subfloat[$C_B(0) = 0.50$,\,\,\,$C_L(0) = 0.25$]{%
\includegraphics[height=10.5cm,width=9cm]{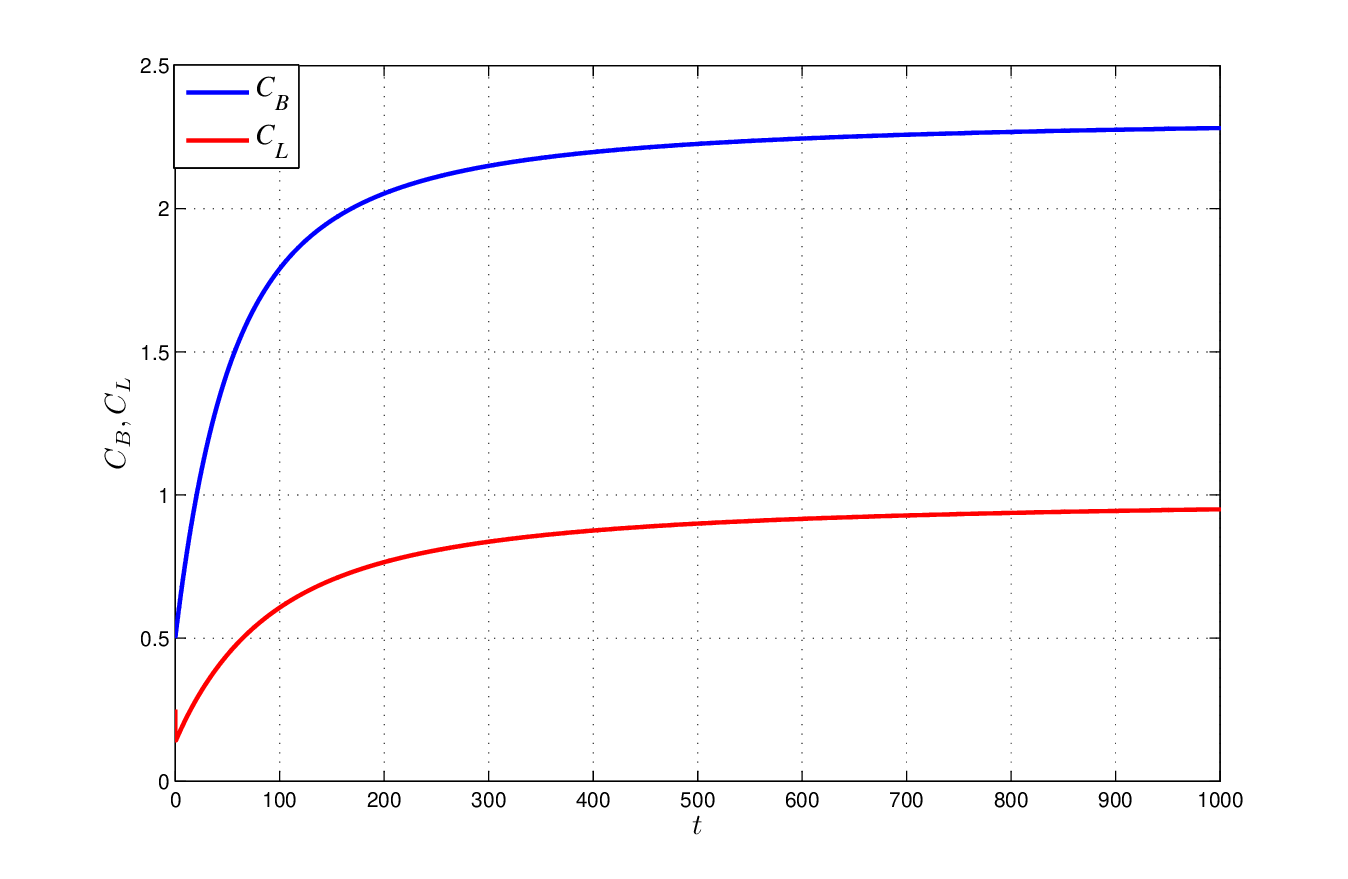}
\label{Figure:1b3}
}\hfill
%%%
%%%
\subfloat[$C_B(0) = 1.50$,\,\,\, $C_L(0) = 1.50$]{%
\includegraphics[height=10.5cm,width=9cm]{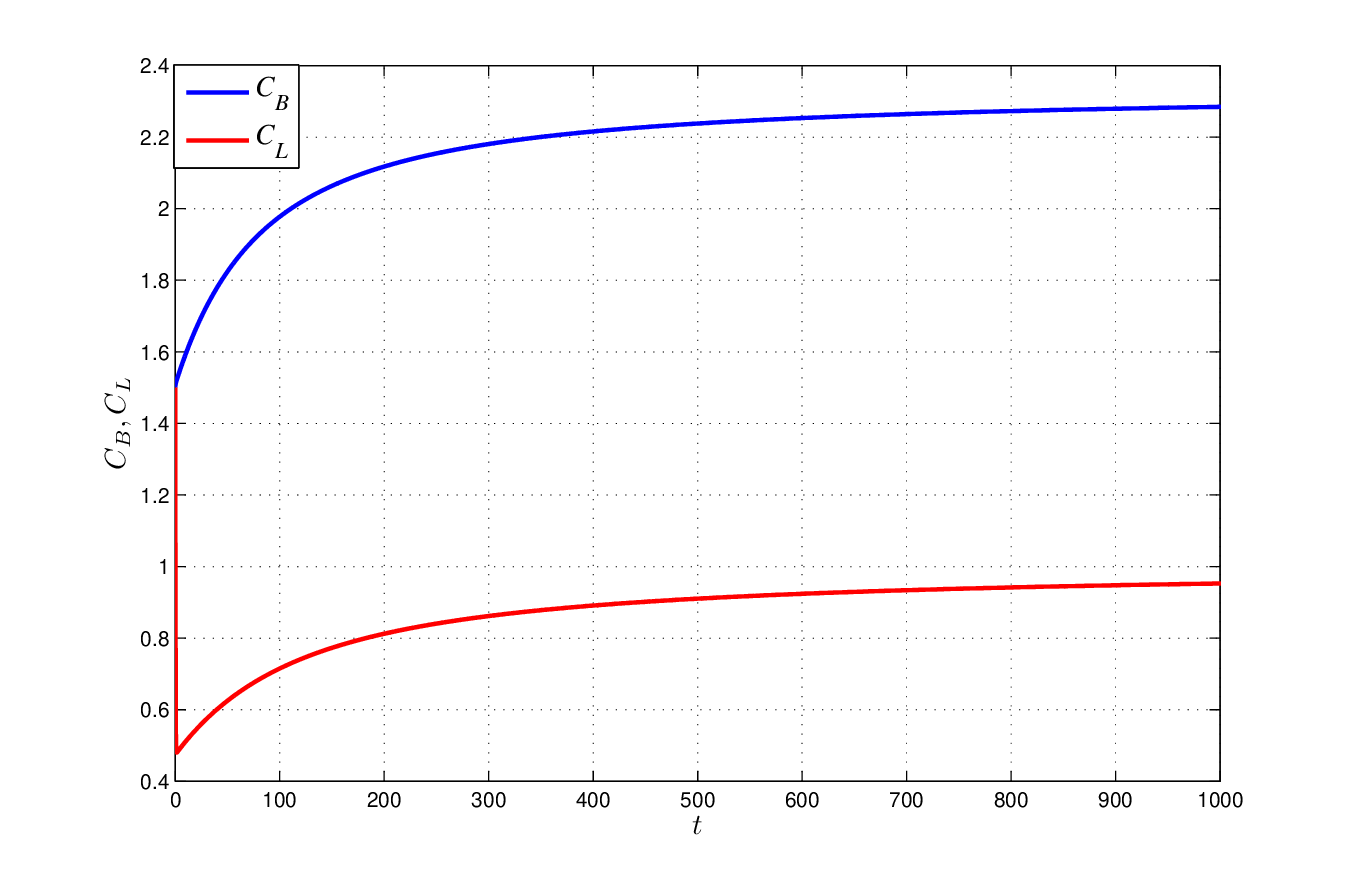}
\label{Figure:1c3}
}\hfill
\subfloat[$C_B(0) = 5.00$,\,\,\,$C_L(0) = 3.00$]{%
\includegraphics[height=10.5cm,width=9cm]{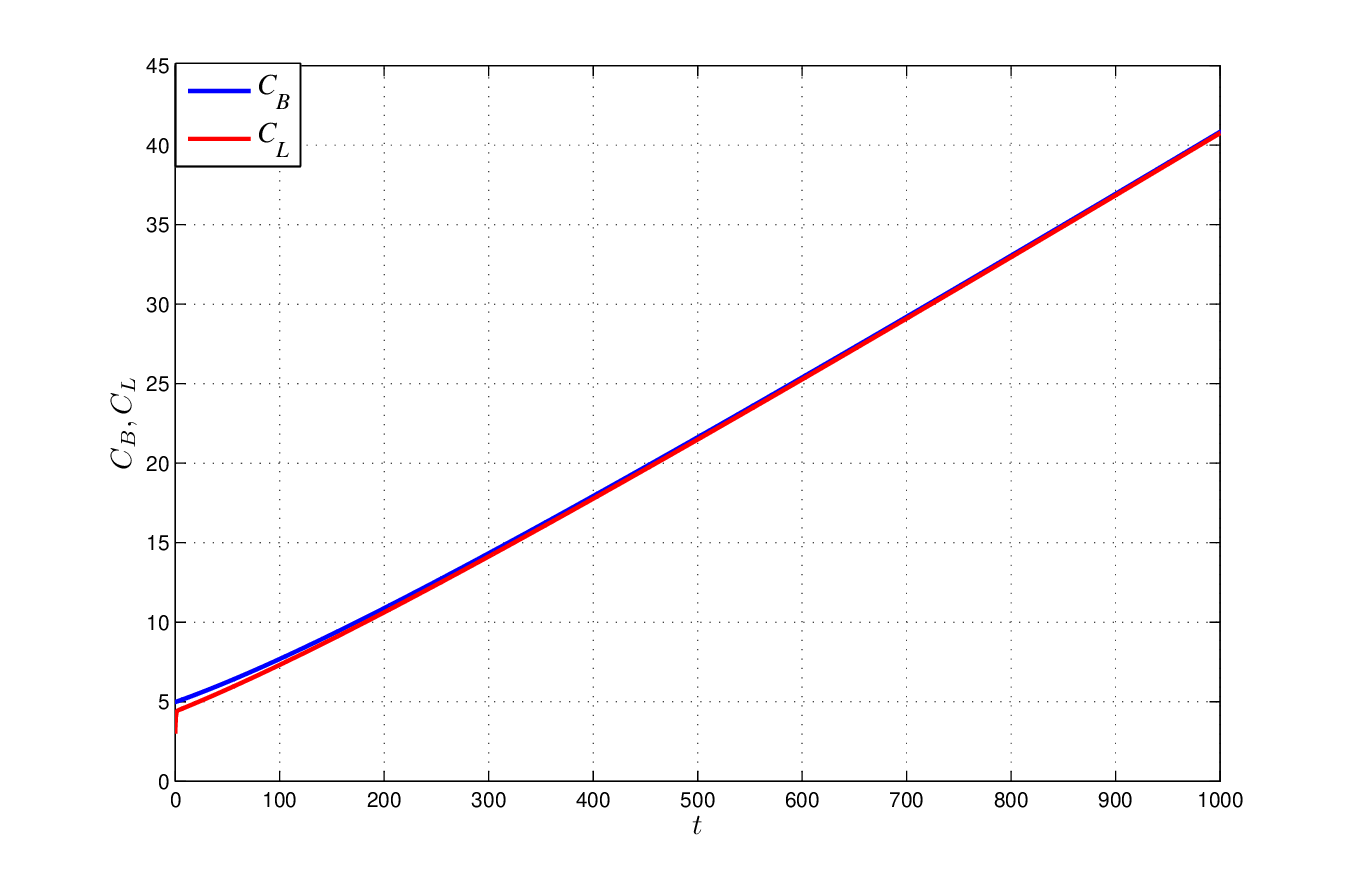}
\label{Figure:1d3}
}
\caption{The solutions of the model \eqref{eq:25} using parameter Set $3$ in Table \ref{Table2}.}
\label{Fig:7}
\end{figure}
\begin{figure}[H]
\subfloat[$C_B(0) = C_L(0) = 0.00$]{%
\includegraphics[height=10.5cm,width=9cm]{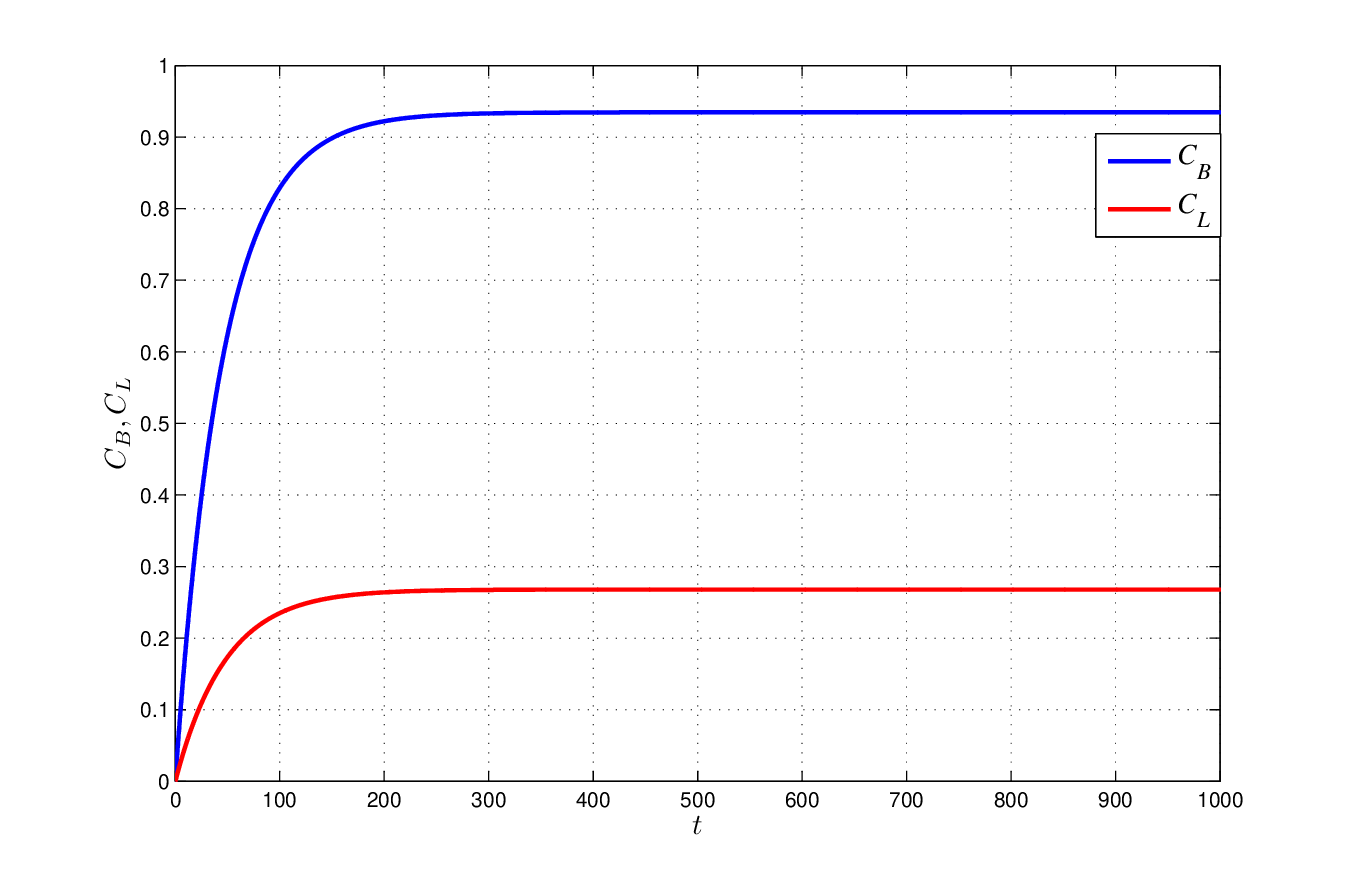}
\label{Figure:1a4}
}\hfill
\subfloat[$C_B(0) = 0.50$,\,\,\,$C_L(0) = 0.25$]{%
\includegraphics[height=10.5cm,width=9cm]{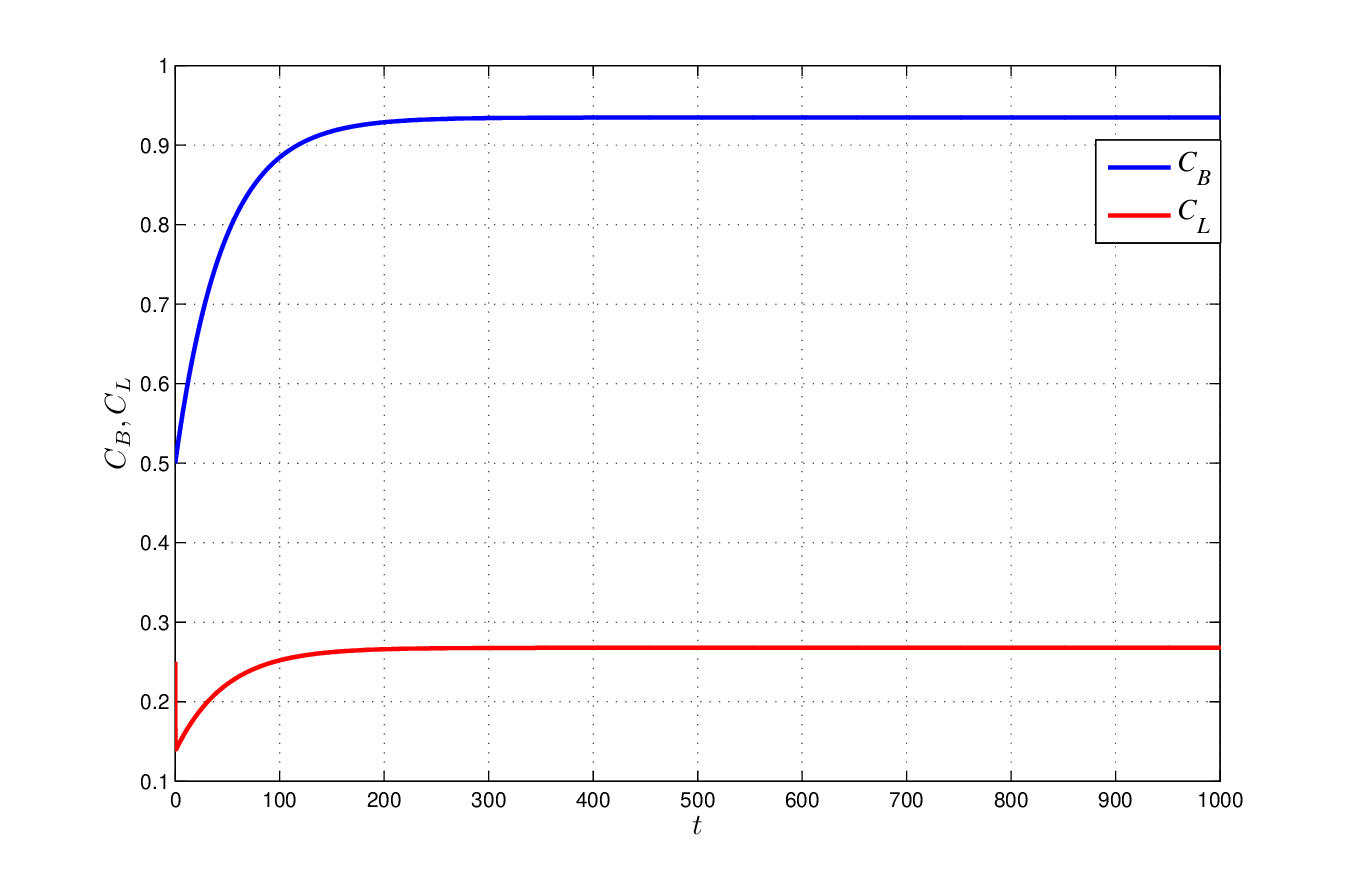}
\label{Figure:1b4}
}\hfill
%%%
%%%
\subfloat[$C_B(0) = 1.50$,\,\,\, $C_L(0) = 1.50$]{%
\includegraphics[height=10.5cm,width=9cm]{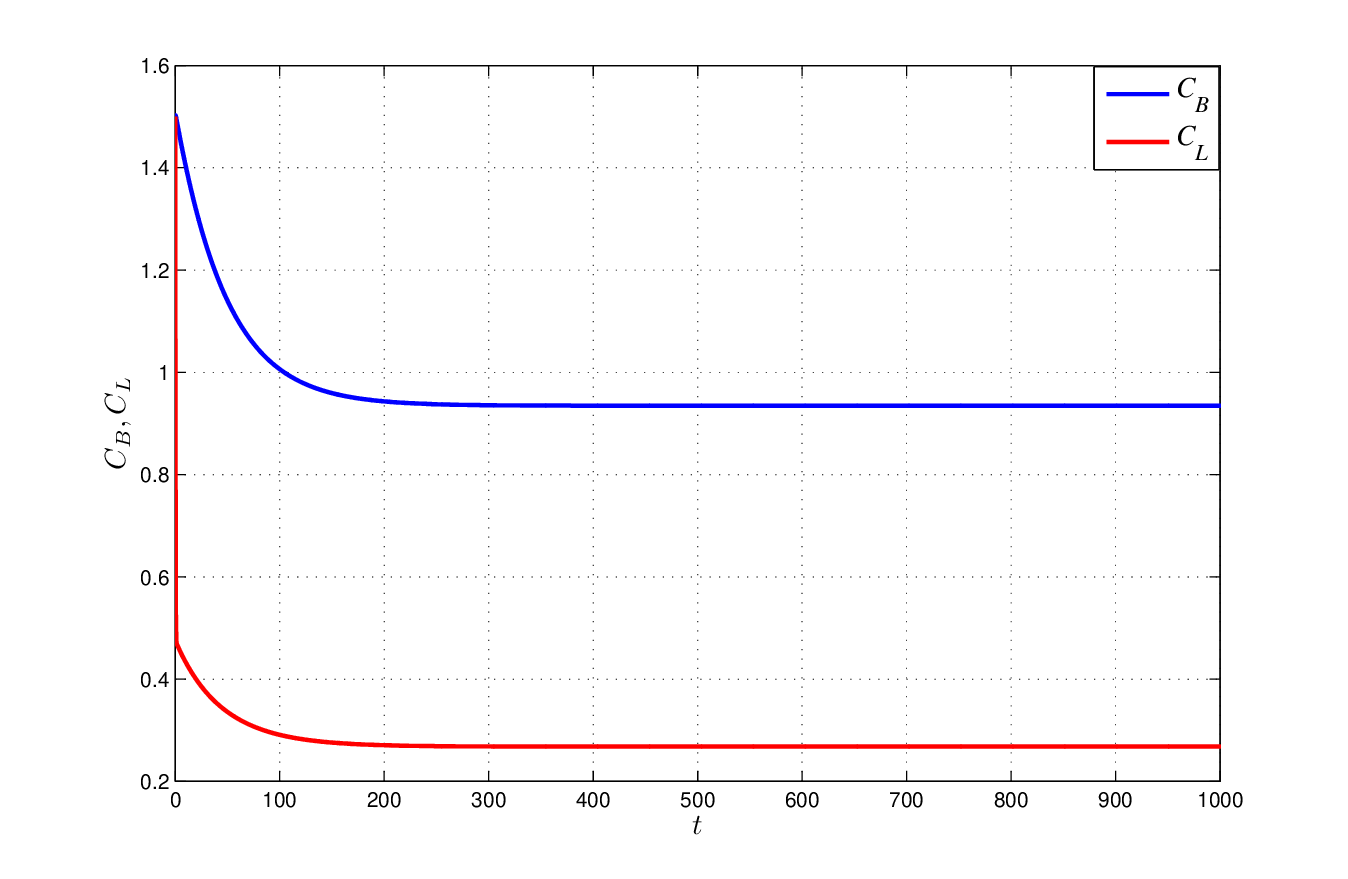}
\label{Figure:1c4}
}\hfill
\subfloat[$C_B(0) = 5.00$,\,\,\,$C_L(0) = 3.00$]{%
\includegraphics[height=10.5cm,width=9cm]{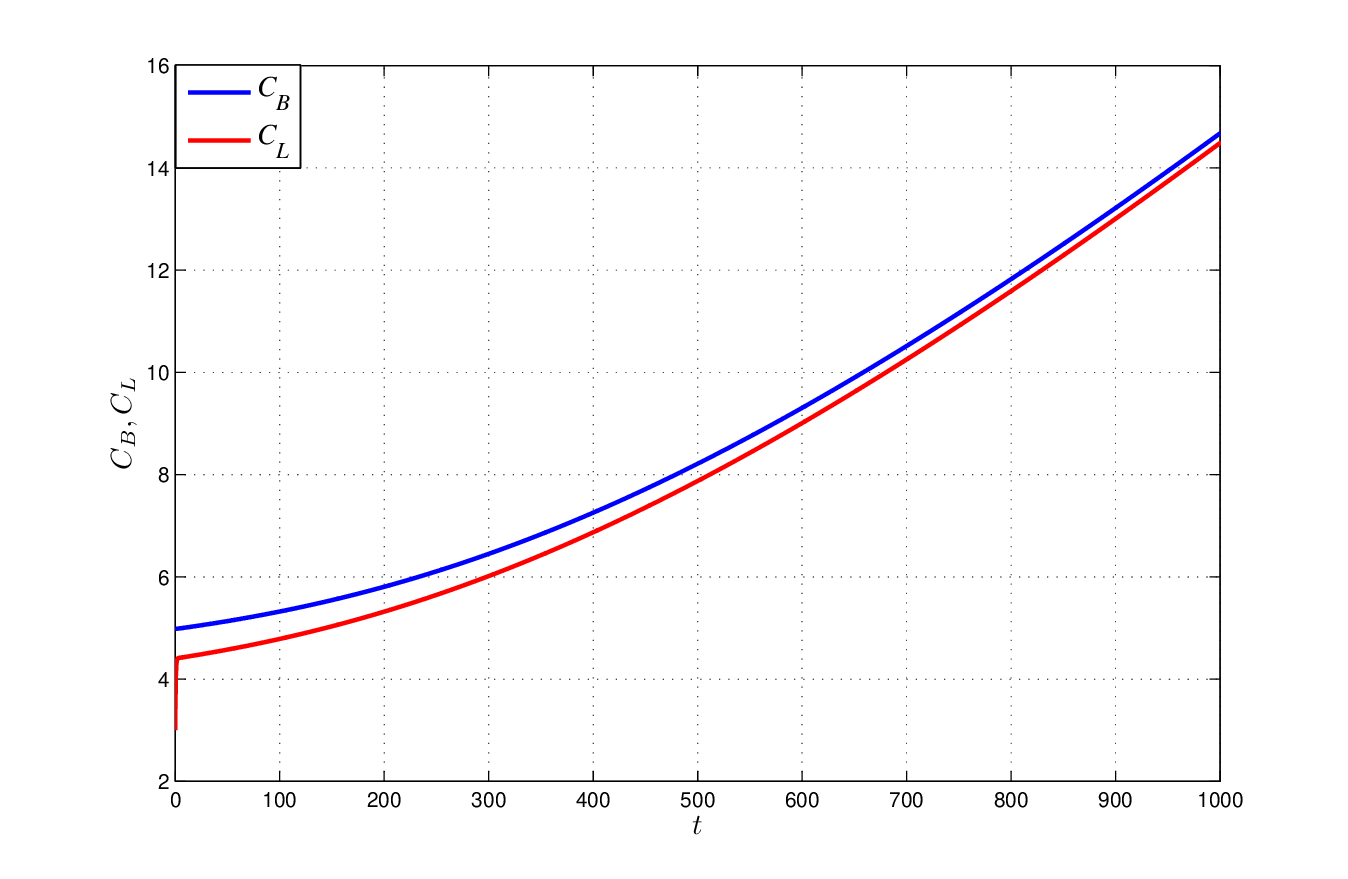}
\label{Figure:1d4}
}
\caption{The solutions of the model \eqref{eq:25} using parameter Set $4$ in Table \ref{Table2}.}
\label{Fig:8}
\end{figure}
\section{Concluding remarks and discussions}
As the first conclusion of this work, we have revisited the well--known continuous--time two--compartment pharmacokinetic model of human ethanol metabolism, originally proposed by Levitt and Levitt  in \cite{Levitt}, and analyzed its global dynamics. In particular, we have established the positivity and boundedness of the solutions, examined the existence and uniqueness of a positive equilibrium, and investigated  its local and global asymptotic stability.

Second, we have extended the original continuous-time model by replacing the Michaelis--Menten metabolism rate with a general class of metabolism-rate functions. This extension enhances the flexibility of the original model and enables it to capture a wider range of realistic metabolic scenarios. We then investigated the global dynamics of the resulting generalized continuous-time model. As a result, the dynamics of the ethanol metabolism model has been completely characterized, thereby complementing and extending the analytical results reported in the original benchmark study.

{The results show that the generalized two-compartment model can exhibit rich dynamics, including multiple equilibria with different stability properties and a saddle-node bifurcation. These findings clarify how ethanol input and metabolic capacity influence long-term behavior and may inform experimental design, parameter estimation, and the assessment of qualitative agreement between model predictions and physiological observations. The analysis also provides a mathematical foundation for future extensions incorporating delays, stochastic effects, fractional-order dynamics, and data-driven approaches.}

Finally, numerical experiments have been conducted to validate the theoretical results. The numerical simulations are consistent with the theoretical analysis and provide strong evidence supporting the theoretical findings.

Future research focuses on applying the proposed theoretical framework to real-world ethanol metabolism data. In addition, the development of efficient numerical methods for the proposed model and its extensions remains an interesting direction for further investigation.\\
%\medskip\noindent
%%
\textbf{Acknowledgment:} We would like to thank the editor and anonymous referees for their useful and valuable comments that led to a great improvement
of the paper.\\
%%%%%%%%%%%%%%
\textbf{Availability of supporting data:} The data supporting the findings of this study are available within the article [and/or] its supplementary materials.\\
\textbf{Conflicts of Interest:} The author declares no conflicts of interest to disclose.\\
\textbf{Authors' contributions:} \textbf{Manh Tuan Hoang} 
Writing review \& editing, Writing original draft, Visualization, Validation, Supervision, Software, Resources, Project administration, Methodology, Investigation,
Formal analysis, Data curation, Conceptualization, Funding acquisition.\\
\textbf{Funding information:} Not available.
{
\begin{appendix}
\section{Center Manifold Analysis and Saddle-Node Bifurcation}\label{AppendixA}
In this appendix, we analyze the non-hyperbolic equilibrium in Theorem~\ref{Theorem5}(ii) for the specific metabolic function. We distinguish between the classification of the critical equilibrium and the bifurcation of equilibria as the total ethanol input changes. Our main tools are the local center manifold theorem (Theorem~1, Section~2.12), the local reduction result (Theorem~2, Section~2.12), and Sotomayor's bifurcation theorem (Theorem~1, Section~4.2) in Perko \cite{Perko2001}; see also \cite{Carr1981, Kuznetsov2023}.
\subsection{The non-hyperbolic equilibrium point}
%%%%%%%%%%%%%%%%%%%%
Consider the two-compartment ethanol metabolism model of Theorem~\ref{Theorem5} with the specific nonlinear metabolic function
\begin{equation}
f(z)=\dfrac{\kappa_1z^n}{\kappa_2+z^m},\qquad \kappa_1,\kappa_2>0,\quad m > n \geq 1.\label{eq:A1}
\end{equation}
For convenience, we write $F=F_{HV}$ and $V=V_B+V_L$. The model can then be expressed as
\begin{equation}\label{eq:A2}
\begin{split}
&V_B\dot C_B =F(C_L-C_B)+Q_{IV},\\
 %%%%%%%%
&V_L\dot C_L =F(C_B-C_L)+Q_{GI}-f(C_L).
\end{split}
\end{equation}
%%%
Differentiating the function $f$ gives
\begin{equation*}
f'(z) = \dfrac{\kappa_1z^{n-1}[n\kappa_2-(m-n)z^m]}{(\kappa_2+z^m)^2}.
\end{equation*}
Consequently, $f$ has a unique positive critical point
\begin{equation*}
s:=\widehat C_L=\left(\frac{n\kappa_2}{m-n}\right)^{1/m},
\end{equation*}
at which it attains its maximum $\widehat V_{\max}=f(s)$. In particular,
\begin{equation*}
 f(s)=\widehat V_{\max},\qquad f'(s)=0.
\end{equation*}
On the other hand, the second derivative of $f$ at $s$ is given by
\begin{equation}\label{eq:app_second}
f''(s)  = -\dfrac{\kappa_1m(m-n)s^{n+m-2}}{(\kappa_2+s^m)^2} = -\dfrac{n(m-n)}{s^2}\widehat V_{\max}<0.
\end{equation}
Note that the identities
$\kappa_2+s^m=\frac{m}{n}s^m$ and
$f(s)=\widehat V_{\max}
=\frac{\kappa_1s^n}{\kappa_2+s^m}$
have been used to derive \eqref{eq:app_second}. Thus, the maximum is nondegenerate. As will be seen below, this property is essential for the center manifold reduction and the saddle-node bifurcation analysis.
%%%%%%%%%%%%%%%%%%%%%%%%%%%%%%%%%%%%%%%%%%%%%%%%%%%%%%%%%%%%%%%%%%%%%%%%%%%%%
We now consider the case
\begin{equation*}
 Q_{IV}+Q_{GI}=\widehat V_{\max}.
\end{equation*}
At the critical parameter value, the unique positive equilibrium is given by
\begin{equation*}
E^* = \left(C_B^*,\,C_L^*\right) = \left(s+\dfrac{Q_{IV}}{F},\,s\right).
\end{equation*}
The Jacobian matrix of \eqref{eq:A2} at $E^*$ is
\begin{equation*}
 J(E^*)=
\begin{pmatrix}
-\dfrac{F}{V_B}&\dfrac{F}{V_B}\\
 %%%%%%%%%%%%%%%%%%
 &\\
 \dfrac{F}{V_L}&-\dfrac{F}{V_L}
 \end{pmatrix},
\end{equation*}
where we have used $f'(s)=0$. Consequently, its eigenvalues of $ J(E^*)$ are
\begin{equation*}
 \lambda_1=0,\qquad \lambda_2=-F\left(\frac1{V_B}+\frac1{V_L}\right)<0.
\end{equation*}
Hence, $E^*$ is non-hyperbolic, and the linearization method is inconclusive regarding its asymptotic stability.
%%%%%%%%%
\subsection{Center Manifold Reduction and Stability Classification}
%%%%%%%%%%%%%%%%%%%%%%%%%%%%%%%%%%%%%%%%%%%%%%%%%%%%%%%%%%%%%%%%%%%%%%%%%%%%%%%%%%%%%%%%%%%%%%%%%%%%%%%%%%%%%%%%%%%%%%%%%%%%%%%%%%%%%%%%%%%%%%%%%%%%%%%%%%%%%%%%%%%%%%%%%%%%%%%%%%%
We use the change of variables
\begin{equation*}
 x=C_B-C_B^*,\qquad y=C_L-s,
\end{equation*}
and define $\phi(y)=f(s+y)-f(s)$. Then, the system \eqref{eq:A2} is rewritten in the form
\begin{equation}\label{eq:System_Shifted}
\begin{split}
\dot{x}& = \dfrac{F}{V_B}(y-x),\\
\dot{y}& = \dfrac{F}{V_L}(x-y)-\dfrac{1}{V_L}\phi(y).
\end{split}
\end{equation}
%%%%%%%%%%%%%%
%%%%%%%%%%%%%%%%
Since $f'(s)=0$ and $f$ is smooth near $s>0$, applying Taylor's formula gives
\begin{equation*}
\phi(y) = \dfrac{1}{2} f''(s)y^2 + \mathcal{O}(y^3) \quad \mbox{as} \quad y \to 0.
\end{equation*}
%%%%%%%%%%%%%%%%%%%%%%
The eigenspace associated with the zero eigenvalue is determined by solving the system $J(E^*)v=0$, which yields
$v_1 = v_2$. Hence,
\begin{equation*}
E^c=\ker J(E^*)
=\operatorname{span}\left\{(1,\,1)^T\right\}.
\end{equation*}
By the Local Center Manifold Theorem
(Perko \cite[Theorem~1, Section~2.12]{Perko2001}),  there exists a local
one-dimensional center manifold through the origin,
tangent to $E^c$, which can be represented as
\begin{equation*}
W^c_{\mathrm{loc}}
=\left\{(x,y):x=y+h(y),\ |y|<\delta\right\},
\label{eq:app_manifold}
\end{equation*}
where $h$ is sufficiently smooth and satisfies $h(0)=h'(0)=0$. Substituting $x=y+h(y)$ into \eqref{eq:System_Shifted} gives
\begin{equation}
 \dot x=-\dfrac{F}{V_B}h(y),\qquad
 \dot y=\dfrac{Fh(y)-\phi(y)}{V_L}.
 \label{eq:System_onmanifold}
\end{equation}
It follows from $\dot x=(1+h'(y))\dot y$ on the center manifold that the invariance equation (cf. equation~(5) in Perko \cite[Section~2.12]{Perko2001}) reads
\begin{equation}
 (1+h'(y))(Fh(y)-\phi(y))=-\dfrac{FV_L}{V_B}h(y).
 \label{eq:System_invariance}
\end{equation}
Writing $h(y)=a_2y^2+O(y^3)$, substituting into \eqref{eq:System_invariance}, and comparing coefficients of $y^2$ yields
\begin{equation*}
 Fa_2 - \dfrac{1}{2} f''(s) = -\dfrac{FV_L}{V_B}a_2.
\end{equation*}
Therefore,
\begin{equation*}
 a_2=\dfrac{V_Bf''(s)}{2F(V_B+V_L)},
 \label{eq:app_a2}
\end{equation*}
and the center manifold has the local expansion
\begin{equation*}
 x=y+\frac{V_Bf''(s)}{2F(V_B+V_L)}y^2+O(y^3).
 \label{eq:app_manifold_expansion}
\end{equation*}
Substitution into the $y$-equation of \eqref{eq:System_onmanifold} gives the reduced scalar equation
\begin{equation*}
 \dot y=-\frac{f''(s)}{2(V_B+V_L)}y^2+O(y^3).
 \label{eq:app_reduced}
\end{equation*}
From \eqref{eq:app_second}, we obtain
\begin{equation}
 \dot y=ay^2+O(y^3),\qquad
 a=\dfrac{n(m-n)\widehat V_{\max}}{2(V_B+V_L)s^2}>0.
 \label{eq:app_normal}
\end{equation}
For all sufficiently small $y\ne0$, the right-hand side of \eqref{eq:app_normal} is positive. Consequently, trajectories on the center manifold with $y<0$ approach the origin in forward time, whereas those with $y>0$ move away from the origin. Thus, the equilibrium is attracting from one side and repelling from the other along the local center manifold. Since the repelling branch contains initial conditions arbitrarily close to $E^*$ (and within the physically relevant positive state space), $E^*$ is \emph{unstable in the sense of Lyapunov}. This conclusion follows already from invariance; it is also consistent with the local dynamical reduction in Perko \cite[Theorem~2, Section~2.12]{Perko2001}.

The reduced equation has a nonzero quadratic
leading term, while the transverse eigenvalue
is strictly negative. Consequently, by the
local reduction theorem
\cite[Section~2.12, Theorem~2]{Perko2001}
and the saddle-node classification
\cite[Section~2.11, Theorem~1]{Perko2001},
$E^*$ is a nondegenerate saddle-node equilibrium.

%Since the reduced equation has a nonzero quadratic leading term and the transverse
%eigenvalue is strictly negative, the equilibrium $E^*$ is of saddle-node type;
%see Perko \cite[Section~2.11, Theorem~1]{Perko2001}.

%The reduced equation has a nonzero quadratic leading term, while the transverse eigenvalue is strictly negative. Consequently, the local reduction theorem \cite[Section~2.12, Theorem~2]{Perko2001},
%combined with the saddle-node classification, implies that $E^*$ is a
%nondegenerate saddle-node equilibrium.
%The leading nonzero term of the reduced equation is quadratic. In particular, the classification theorem for planar analytic systems with one zero eigenvalue (Perko \cite[Theorem~1, Section~2.11, p.~151]{Perko2001}) applies, since $s>0$ and the present metabolic vector field is real analytic in a neighborhood of $E^*$. Because the leading exponent is even, $E^*$ is a \emph{saddle-node equilibrium}. The nonzero quadratic coefficient also establishes its nondegeneracy.
%
%
%
%
%
\subsection{Saddle-node bifurcation}
We next establish the occurrence of a nondegenerate saddle-node bifurcation at the critical equilibrium as the total
ethanol input rate passes through $\widehat V_{\max}$. First, we define
\begin{equation*}
 Q=Q_{IV}+Q_{GI}. %\mu=Q-\widehat V_{\max}.
 \label{eq:app_mu}
\end{equation*}
To specify a one-parameter family, fix $\theta\in[0,1]$ and set
\begin{equation*}
 Q_{IV}=\theta Q,\qquad Q_{GI}=(1-\theta)Q.
 \label{eq:app_parameterization}
\end{equation*}
Thus, the relative contributions of the two inputs remain fixed as $Q$ varies. Let $X=\left(C_B,\,C_L\right)^T$ and write $\dot X=\mathcal F(X,Q)$, where
\begin{equation*}
 \mathcal F(X,Q)=\begin{pmatrix}
 \dfrac{F(C_L-C_B)+\theta Q}{V_B}\\
 \\
 \dfrac{F(C_B-C_L)+(1-\theta)Q-f(C_L)}{V_L}
 \end{pmatrix}.
 \label{eq:app_F}
\end{equation*}
At $Q=\widehat V_{\max}$, the Jacobian $D_X\mathcal F(E^*,Q)$ has a simple zero eigenvalue. Corresponding right and left null eigenvectors, normalized by $\ell^Tr=1$, are
\begin{equation*}
 r=\begin{pmatrix}1\\1\end{pmatrix},\qquad
 \ell=\frac{1}{V_B+V_L}\begin{pmatrix}V_B\\V_L\end{pmatrix}.
 \label{eq:app_nullvectors}
\end{equation*}
The parameter derivative satisfies
\begin{equation}
 \mathcal F_Q(E^*,\widehat V_{\max})=
 \begin{pmatrix}\dfrac{\theta}{V_B}\\
 \dfrac{1-\theta}{V_L}\end{pmatrix},
 \qquad
 \ell^T\mathcal F_Q(E^*,\widehat V_{\max})=\frac{1}{V_B+V_L}\ne0.
 \label{eq:app_transversality}
\end{equation}
Moreover, the second derivative in the direction $r$ is
\begin{equation*}
 D_X^2\mathcal F(E^*,\widehat V_{\max})(r,\,r)
 =\begin{pmatrix}0\\-\dfrac{f''(s)}{V_L}\end{pmatrix}.
 \label{eq:app_hessian}
\end{equation*}
Consequently,
\begin{equation}
 \ell^T\left[D_X^2\mathcal F(E^*,\widehat V_{\max})(r,\,r)\right] =-\dfrac{f''(s)}{V_B+V_L} =\dfrac{n(m-n)\widehat V_{\max}}{(V_B+V_L)s^2}>0.
\label{eq:app_nondegeneracy}
\end{equation}
We now use Sotomayor's theorem (Perko \cite[Theorem~1, Section~4.2, Conditions~(2)]{Perko2001}; see also \cite{Kuznetsov2023}). The equations \eqref{eq:app_transversality} and \eqref{eq:app_nondegeneracy} verify exactly its parameter-transversality condition and quadratic nondegeneracy condition, respectively. The zero eigenvalue is simple, while the other eigenvalue is strictly negative. Thus, all hypotheses needed for a generic nondegenerate saddle-node bifurcation hold, and the bifurcation occurs at $Q=\widehat V_{\max}$.

In particular, for $Q<\widehat V_{\max}$ sufficiently close to the threshold, the system has two nearby positive equilibria: the equilibrium with smaller liver concentration is locally asymptotically stable, whereas the other is a saddle. They merge at $Q=\widehat V_{\max}$ and disappear for $Q>\widehat V_{\max}$, in agreement with Theorem~\ref{Theorem5}.

In conclusion, for the specific metabolic function \eqref{eq:A1}, the critical equilibrium in Theorem~\ref{Theorem5}(ii) is a nondegenerate saddle-node, is Lyapunov unstable, and undergoes a nondegenerate saddle-node bifurcation as $Q_{IV}+Q_{GI}$ crosses $\widehat V_{\max}$.
\end{appendix}
}
{
\section{Ruling Out Periodic Orbits}\label{AppendixB}
We further investigate the global dynamics of the generalized system, even when the metabolic function \(f(C_L)\) is non-monotone. In particular, we rule out the existence of periodic orbits of the generalized system.

Consider the model \eqref{eq:1new} with
$V_B,V_L,F_{HV}>0$, $Q_{IV},Q_{GI}\geq 0$, $Q=Q_{IV}+Q_{GI}>0$ and $f$ satisfies the property \textbf{(A)}. Then, we will show that it admits no nonconstant periodic orbits in the positive quadrant regardless of the monotonicity of $f$.

Indeed, it is easy to see that \eqref{eq:1new} can be rewritten in the form:
\begin{equation*}
\begin{split}
& -V_B \dot{C}_B=F_{H V}\left(C_B-C_L\right)-Q_{I V}, \\
& -V_L \dot{C}_L=F_{H V}\left(C_L-C_B\right)-Q_{G I}+f\left(C_L\right).
\end{split}
\end{equation*}
We seek a continuously differentiable function
$\mathcal L(C_B,C_L)$ that satisfies
\begin{equation*}
\begin{split}
\dfrac{\partial\mathcal L}{\partial C_B}
&= -V_B \dot{C}_B = F_{HV}(C_B-C_L)-Q_{IV},\\
\dfrac{\partial\mathcal L}{\partial C_L}
&= -V_L \dot{C}_L = F_{HV}(C_L-C_B)-Q_{GI}+f(C_L).
\end{split}
\end{equation*}
%Note that the compatibility condition is satisfied:
%\begin{equation*}
%\dfrac{\partial}{\partial C_L}
%\left(\dfrac{\partial\mathcal L}{\partial C_B}\right)
%=
%\dfrac{\partial}{\partial C_B}
%\left(\dfrac{\partial\mathcal L}{\partial C_L}\right)
%=-F_{HV}.
%\end{equation*}
Integrating the first relation with respect
to $C_B$ yields
\begin{equation*}
\mathcal L(C_B,C_L)
=\frac{F_{HV}}{2}(C_B-C_L)^2
-Q_{IV}C_B + g(C_L),
\end{equation*}
where $g$ is a function of $C_L$.
Substitution into the second relation gives
$g'(C_L) = f(C_L)-Q_{GI}$.
Consequently, up to an additive constant,
we obtain
\begin{equation*}
\mathcal L(C_B,C_L)
=\frac{F_{HV}}{2}(C_B-C_L)^2
-Q_{IV}C_B-Q_{GI}C_L
+\int_0^{C_L}f(s)\,ds.
\label{eq:B1}
\end{equation*}
It is clear that
\begin{equation*}
\dfrac{\partial\mathcal L}{\partial C_B}
=-V_B\dot C_B,
\qquad
\dfrac{\partial\mathcal L}{\partial C_L}
=-V_L\dot C_L.
\end{equation*}
Hence, along every solution, we have
\begin{equation}
\dfrac{d\mathcal L}{dt} = \dfrac{\partial\mathcal L}{\partial C_B} \dot C_B + \dfrac{\partial\mathcal L}{\partial C_L}\dot C_L = -V_B(\dot C_B)^2 - V_L(\dot C_L)^2\leq0,
\label{eq:B2}
\end{equation}
where equality holds if and only if the state is an equilibrium.
Thus, $\mathcal L$ decreases strictly along
every nonstationary trajectory. Suppose, for contradiction, that there exists
a nonconstant periodic solution with period
$T>0$. Then, by \eqref{eq:B2}, we obtain
\begin{equation*}
0 = \mathcal L\left(C_B(T),\,C_L(T)\right) - \mathcal L\left(C_B(0),\,C_L(0)\right) = \int_0^T\dfrac{d\mathcal L}{dt} dt = -\int_0^T
\left[V_B(\dot C_B)^2 +  V_L(\dot C_L)^2\right]dt < 0,
\end{equation*}
which is a contradiction. Therefore, the system admits no nonconstant periodic orbits.
}
%%
%%
%%

% \clearpage

%%%%%%%%%%%%%%%%%%%%%%%%%%%%%%%%%%%%%%%%%%%%%%%%%%%%%%%% References
\bibliographystyle{amsalpha}

\end{document}